\documentclass[12pt]{amsart}
\usepackage[utf8]{inputenc}
\usepackage[letterpaper,total={6.5in,9in},top=1in, left=1in, right=1in, bottom=1in]{geometry}
\usepackage{fancyhdr}
\usepackage{libertine}
\usepackage[libertine]{newtxmath}
\usepackage{color,amsmath,mathtools,graphicx}
\usepackage{bbold}
\usepackage{epsfig,fancyhdr}
\usepackage{dsfont} 
\usepackage{graphicx,amsmath,mathtools,float}
\usepackage{epsfig,color,fancyhdr,setspace}
\usepackage{dsfont}
\usepackage{epstopdf}
\usepackage{import}
\usepackage{lineno}
\usepackage{stackrel,dsfont}
\usepackage[allcolors = blue,colorlinks]{hyperref}
\usepackage[abs]{overpic}
\usepackage[center]{caption}
\usepackage{subcaption}
\usepackage{tikz-cd}
\usepackage{enumitem}
\counterwithin{figure}{section}
\usepackage[utf8]{inputenc}
\usepackage[T1]{fontenc}
\usepackage{ragged2e}
\usepackage{extarrows}
\usepackage{rotating}
\allowdisplaybreaks
\newtheorem{theorem}{Theorem}[section]
\newtheorem{lemma}[theorem]{Lemma}
\newtheorem{definition}[theorem]{Definition}
\newtheorem{example}[theorem]{Example}
\newtheorem{proposition}[theorem]{Proposition}
\newtheorem{remark}[theorem]{Remark}
\newtheorem{corollary}[theorem]{Corollary}

\newtheorem{innercustomgeneric}{\customgenericname}
\providecommand{\customgenericname}{}
\newcommand{\newcustomtheorem}[2]{%
  \newenvironment{#1}[1]
  {%
   \renewcommand\customgenericname{#2}%
   \renewcommand\theinnercustomgeneric{##1}%
   \innercustomgeneric
  }
  {\endinnercustomgeneric}
}

\newcustomtheorem{customthm}{Theorem}
\newcustomtheorem{customlemma}{Lemma}

\title{KAUFFMAN BRACKET SKEIN HOMOLOGY FROM HEEGAARD SPITTINGS}

\author{Yin Tian}
\address{School of Mathematical Sciences, Beijing Normal University; 
Laboratory of Mathematics and Complex Systems, Ministry of Education, Beijing 100875, China}
\email{{\rm yintian@bnu.edu.cn}}

\author{Xiao Wang} 
\address{Jilin University, Changchun, China}
\email{\rm wangxiaotop@jlu.edu.cn}

\keywords{Kauffman bracket skein module, Heegaard splitting, quantum topology.}
\subjclass[2020]{Primary: 57K31. Secondary: 57K10}

\begin{document}

\maketitle

\begin{abstract}
We extend the Kauffman bracket skein module of three manifolds to a homology theory in a combinatorial way. The resulting homology depends on Heegaard splittings of three manifolds.     
\end{abstract}

\section{Introduction}
\iffalse
History of KBSM; P's surjection map from 2-handle attachments. 

Natural to ask whether it is possible to extend this map to a chain complex. 
\fi

The Kauffman bracket skein module theory was introduced independently by Przytycki\cite{smof3} and Turaev\cite{turaevsolidtorus} around $1990$. It is an invariant of three dimensional manifolds which is deeply related to $SL(2,\mathbb{C})$ character varieties, and the quantum Techemuller space. Some of the early researches are listed here \cite{Tur2, Lic, KL, BHM, Rob, Gel, sl2cbullock, saofsurfaces, fundamentals, connsum}. In recent years, it has received extra attention due to the work in \cite{BW, FKL, GJS}. Especially, Gunningham, Jordan, Safronov confirmed the Witten's finiteness conjecture for Kauffman bracket skein modules in \cite{wittenresolved}, i.e. the Kauffman bracket skein module of any closed oriented $3$-manifold is finite dimensional over $\mathbb Q(A)$. 
This makes it possible to build a $3+1$-topological quantum field theory with the Kauffman bracket skein module as the vector spaces associated to the closed three manifolds.

In \cite{fundamentals}, Przytycki developed a method to compute the Kauffman bracket skein module as follows. Suppose a three manifold is given by adding $2$-handles and $3$-handles to a handlebody $H_g$.  When adding a collection of $2$-handles along  simple closed curves  $\mu$ on the boundary of $H_g$, we obtain a three manifold $(H_{g})_{\mu}$. Then the homomorphism on the Kauffman bracket skein modules induced by the embedding $H_{g}\hookrightarrow (H_{g})_{\mu}$ is surjective with the kernel given by the handle sliding relations. Furthermore, adding a $3$-handle does not change the Kauffman bracket skein module. Thus, the Kauffman bracket skein module of the three manifold is isomorphic to the quotient of the Kauffman bracket skein module of $H_{g}$ by the handle sliding relations introduced by the $2$-handle attaching.\footnote{This is also true when we replace $H_g$ by an arbitrary three manifold with boundary.}  This gives a presentation of the Kauffman bracket skein module of a three manifold via its handle decomposition, or a Heegaard splitting.  

In this article, we extend the Kauffman bracket skein module to a homology theory via handle decompositions. More precisely, we extend the surjective homomorphism on the Kauffman bracket skein modules induced by $H_{g}\hookrightarrow (H_{g})_{\mu}$ to a chain complex. So the zeroth homology is the Kauffman bracket skein module of $M$ computed in the way as above.

Consider a Heegaard diagram $(\Sigma,\alpha,\beta)$ of $M$, where $\Sigma$ is the Heegaard surface, and $\alpha, \beta$ denote the attaching curves for the two handlebodies.  
Let $H$ denote the handlebody associated to $\alpha$, and $\mu$ denote $\beta$. We write the Heegaard diagram as $(H,\mu)$. So an isotopy class of Heegaard splittings is the same as an equivalence class of Heegaard diagrams $(H,\mu)$ up to isotopies and handleslides of $\mu$. 
To any Heegaard diagram $(H,\mu)$, we associate a chain complex $C_{*}(H,\mu)$, called the {\em Heegaard skein complex}. Its homology is called the {\em Heegaard skein homology}. Our main result is the following. 

\begin{theorem}\label{main}
The isomorphism class of the Heegaard skein complexes is invariant under isotopies and handleslides of $(H,\mu)$. 
\end{theorem}

\begin{corollary}
The isomorphism class of the homology of the Heegaard skein complexes is an invariant of Heegaard splittings. 
\end{corollary}

\begin{remark}\label{surface}
    Our definition of the Heegaard skein complex is based on a handlebody and $2$-handle attaching.  It is possible to define the Heegaard skein complex starting from thickened surface and attach $\alpha$ curves and $\beta$ curves.  This version is also invariant under isotopies and handleslides, of which the proof is transparent in a similar way.
\end{remark}

\iffalse
\begin{remark}
    Though in Theorem \ref{main} and Remark \ref{surface}, we are considering the Kauffman bracket skein chain complex and its homology, for other skein modules, our definition also works and the two homology theories are invariant under isotopies and handleslides since all skein relations happen locally.
\end{remark}
\fi

%\footnote{For those three manifolds with finitely many minimal genus Heegaard splittings, we can do enough stabilizations so that our homology only depends on the three manifolds themselves. However, the computational complexity will be quite high.}

Our construction of the Heegaard skein complex $C_{*}(H,\mu)$ is combinatorial. 
The chain group $C_{n}(H,\mu)$ is the {\it relative} Kauffman bracket skein module of the handlebody $H$ with $2n$ points on the boundary. The differential is induced by completing relative skeins to skeins through the attaching curves $\mu$. 
As a result, $C_{n}(H,\mu)$ is nonzero for all $n \ge 0$, except for the genus zero Heegaard splitting of $S^3$. We expect that Heegaard skein homology is also nonzero for all $n$. 

Computation of Heegaard skein homology is not easy. 
We compute the first homology associated to the genus one Heegaard splittings of the lens spaces $L(p,1)$. See Proposition \ref{p1}. 
For $p=1$, we get a genus one Heegaard splitting of $S^3$ as a stabilization of the genus zero splitting. The first homologies before and after the stabilization are not isomorphic. 

\begin{corollary}
The Heegaard skein complex is in general not invariant under the stabilization of Heegaard splittings.
\end{corollary}

So the Heegaard skein homology is not a topological invariant of three manifolds. 
Although we lose the topological invariance, we expect that the Heegaard skein homology can be used to distinguish Heegaard splittings of a given three manifold. 

The Kauffman bracket skein module is closely related to a quantization of the $SL(2,\mathbb{C})$-character variety\cite{sl2cbullock}. By specializing the quantum variable $A=-1$, the chain group $C_{n}(H,\mu)$ can be viewed as the coordinates ring of the moduli space of {\it framed} flat $SL(2,\mathbb{C})$-bundles over the handlebody $H$ with framings at $2n$ points. In particular, $C_{n}(H,\mu)$ only depends on $H$. 
The differential is induced by an operation of extending a bundle to a framed bundle through the parallel transportation along the $2$-handle attaching curves $\mu$. 

Abouzaid-Manolescu developed a homology theory generalizing the Kauffman bracket skein module \cite{AM}. The construction uses Heegaard splittings and Floer theory. Their homology is a topological invariant. It is interesting to explore the possible connections between their homology and ours.   

This article is organized as follows. In Section $2$, we recall the definition and basic properties of the Kauffman bracket skein module. In Section $3$, we give the definition of the Heegaard skein complex and prove its invariance under all the choices made when defining the chain complex. In Section $4$, we show that the Heegaard skein complex is invariant under $2$-handle sliddings. In Section $5$, we discuss how the Heegaard skein complex changes under a stabilization. In Section $6$, we compute the Heegaard skein homology associated to the genus one Heegaard splttings of the lens spaces $L(p,1)$.

\iffalse
we can instead consider the minimal genus Heegaard splittings up to homeomophism of a given three manifold $M$. 
Let $MHS(M)$ denote the set of homeomorphism classes of the minimal genus Heegaard splittings of $M$. 
A key fact is that $MHS(M)$ is a finite set.

The main goal is to study the collection of finitely many Heegaard skein complexes
$$\{C_*(s), s \in MHS(M)\}.$$

\begin{remark}
All Heegaard splittings are assummed to be of minimal genus.
\end{remark}

{\em Questions:} 
\begin{enumerate}
\item Explicit computations of $C_*(s)$ for lens spaces, Seifert manifolds of genus two.
\item Compare different Heegaard splittings. The minimal genus is two due to Waldhausen's Uniqueness Theorem which states that any lens space has a unique Heegaard splitting of genus one. There are examples of small Seifert manifolds of genus two.
\item Can we use it to distinguish three manifolds? e.g. For lens spaces $L(p,q)$, the skein modules only depend on $p$ (is this correct? Say $L(5,1)$ and $L(5,2))$; is it possible that $H_1$ can be used to distinguish them?
\item The homology of $C_*(s)$ is bounded? finite rank?
\item Is there a way to quotient out the trivial pairs in $H_1$?
\item What are interesting / natural chain maps? So far, we only have such a chain map induced by adding 2-handles. 
\item Hidden symmetry from algebra? e.g. is the chain complex modules over some algebras?
\end{enumerate}

\fi

\section{preliminary}
In this section, we recall the definitions of the Kauffman bracket skein module and the relative Kauffman bracket skein module.

\begin{definition}\label{kbsmdef}

Let $M$ be an oriented $3$-manifold, $R$ a commutative ring with unity, and $A \in R$ a fixed invertible element. Consider  the set of ambient isotopy classes of unoriented framed links (including the empty link $\varnothing$) in $M$, which we denote by $\mathcal{L}^{\mathit{fr}}$,  and the free $R$-module with basis $\mathcal{L}^{\mathit{fr}}$, denoted by $R\mathcal{L}^{\mathit{fr}}$.
Let $S_{2, \infty}^{\mathit{sub}}$ be the submodule of $R\mathcal{L}^{\mathit{fr}}$ generated by the following expressions: 

\begin{enumerate}

    \item the Kauffman bracket skein expression: $L_+ - AL_0 - A^{-1}L_{\infty}$ and
    
    \item the trivial component expression: $L \sqcup {\pmb \bigcirc}  + (A^2 + A^{-2})L$,
    
\end{enumerate}
\noindent
where $\pmb{\bigcirc}$ denotes the trivial framed knot in $M$ and the skein triple $(L_+$, $L_0$, $L_{\infty})$ denotes three framed links in $M$, which are identical except in a small $3$-ball in $M$ where they differ as illustrated in Figure \ref{skeintriple}.

\begin{figure}[ht]
    \centering
\[  \begin{minipage}{1.3 in} \includegraphics[width=\textwidth]{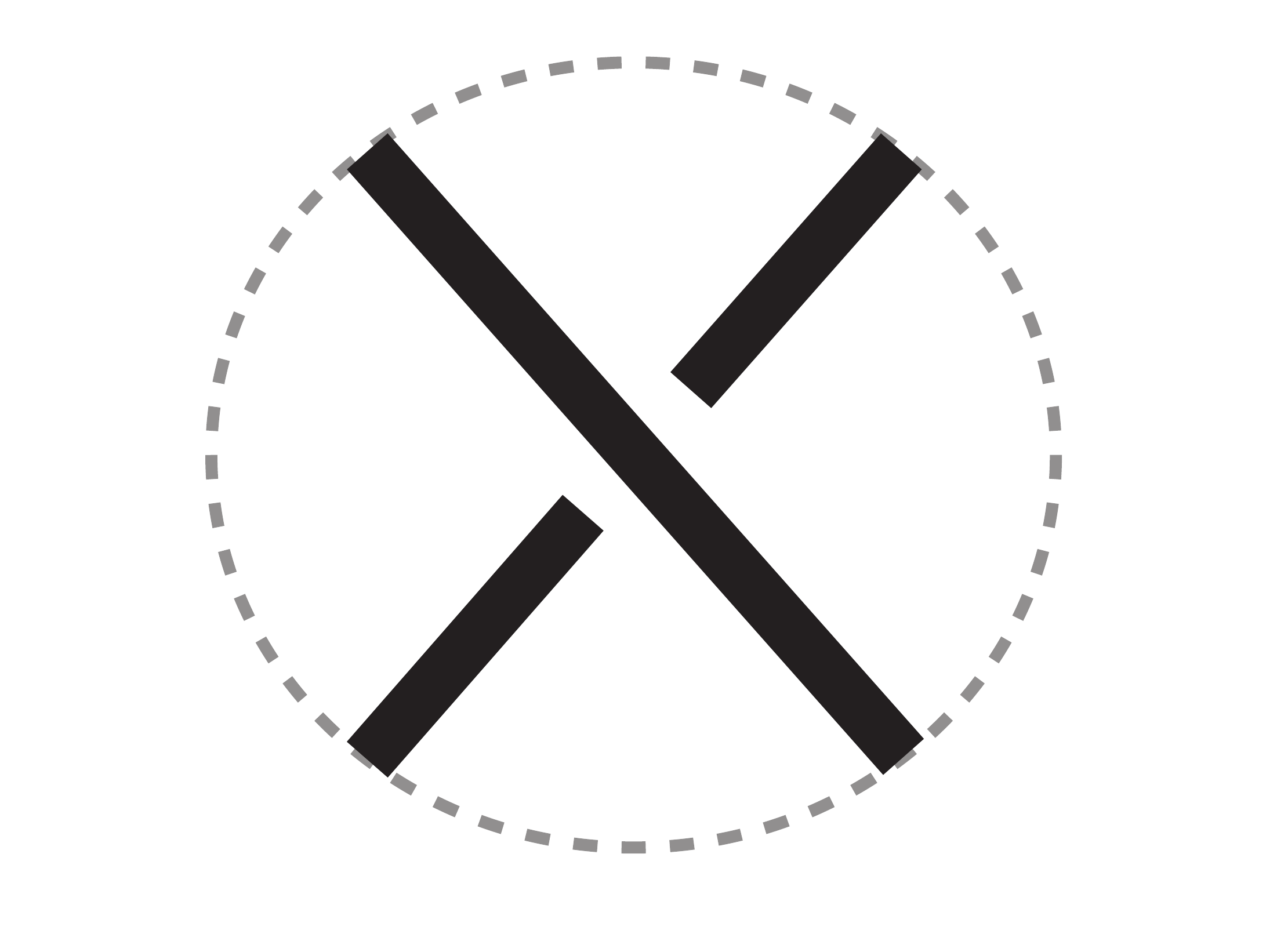} \vspace{-15pt} \[L_+\] \end{minipage} 
               \qquad
        \begin{minipage}{1.3 in}\includegraphics[width=\textwidth]{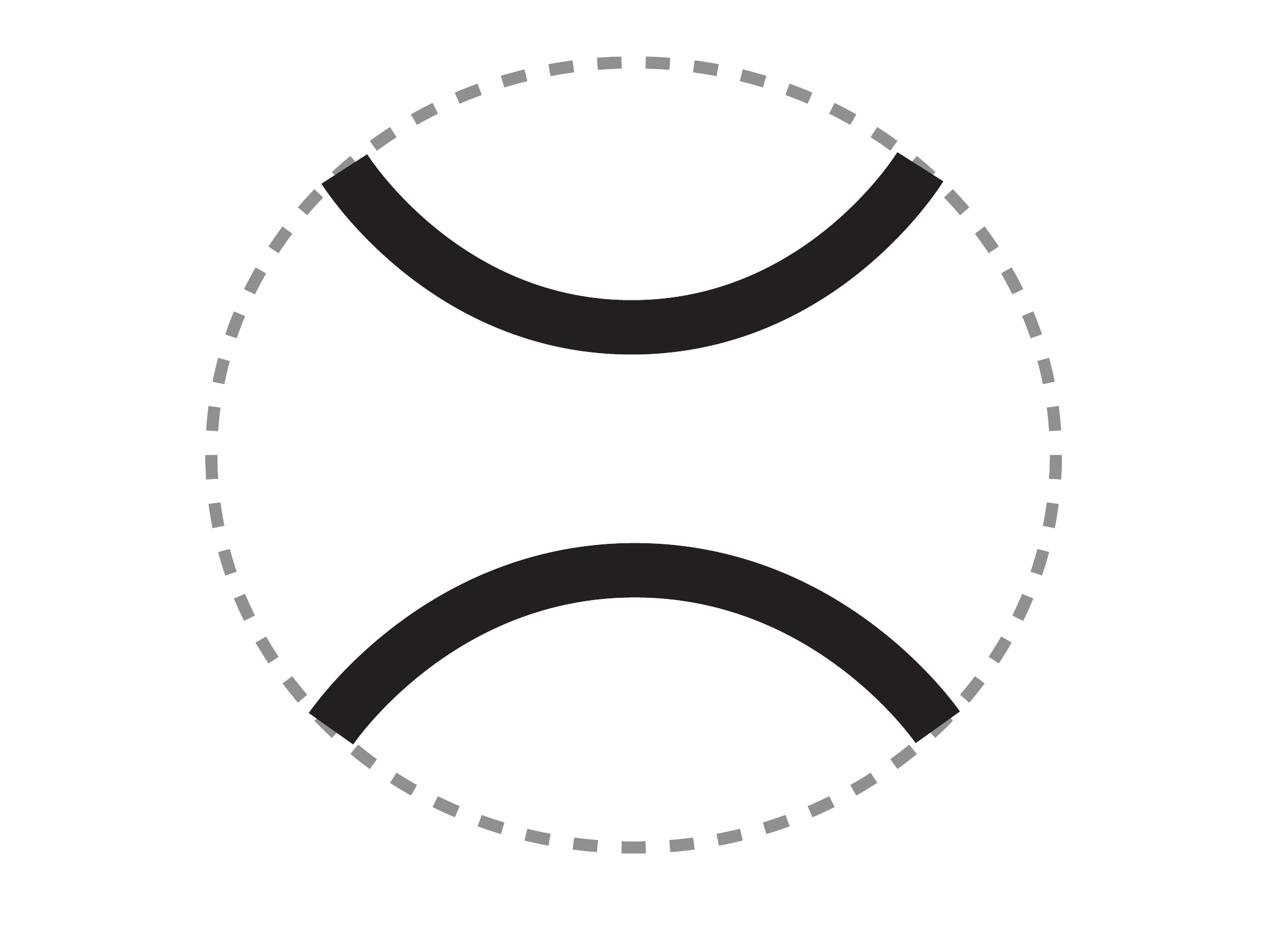} \vspace{-15pt} \[L_0\] \end{minipage}
         \qquad
        \begin{minipage}{1.3 in}\includegraphics[width=\textwidth]{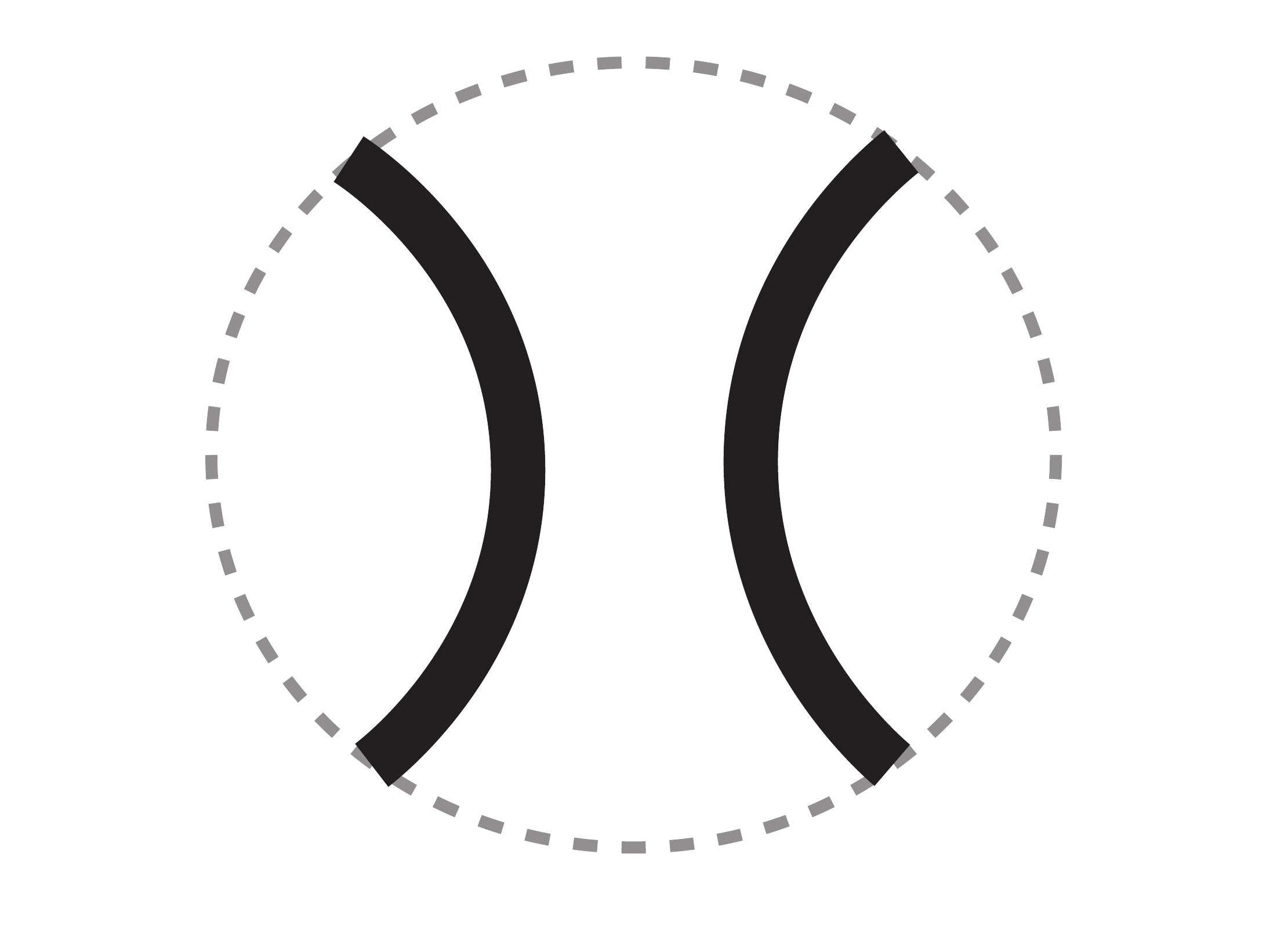} \vspace{-15pt} \[ L_\infty\]\end{minipage} 
        \]
\caption{Skein triple for the Kauffman bracket skein module.}
          \label{skeintriple}
        \end{figure}

The {\bf Kauffman bracket skein module} of $M$ is defined as the quotient: $$\mathcal{S}_{2,\infty}(M;R,A) = \frac{R\mathcal{L}^{\mathit{fr}}}{S_{2, \infty}^{sub}}.$$

\end{definition}

We can also define a relative version of the Kauffman bracket skein module for oriented $3$-manifolds that have framed points on their boundaries (see \cite{smof3, fundamentals}).  

%\begin{example}\cite{sm&j}\label{kbsms3} $\mathcal{S}_{2,\infty}(S^3) = \mathbb{Z}[A^{\pm 1}]\varnothing$. More precisely, $\varnothing$ is the basis element of the module and $L = [L] \varnothing = (-A^2 - A^{-2})\langle \ L \ \rangle \varnothing $, where $[ L ]$ is the unreduced Kauffman bracket polynomial of a framed link $L$. Moreover, $\mathcal{S}_{2,\infty}(B^3) \cong \mathcal{S}_{2,\infty}(\mathbb R^3) = \mathbb{Z}[A^{\pm 1}]\varnothing$.
    
%\end{example}

\begin{definition}

Let $(M,\partial M)$ be an oriented $3$-manifold, $\{x_i\}_{1}^{2n}$ be a set of $2n$ oriented framed\footnote{A framed point in $\partial M$ is an interval in $\partial M$. Thus, a relative framed link intersects
$\partial M$ at framed points.} points on $\partial M$, and $R$ be a commutative ring with unity with a fixed invertible element $A$. Let $\mathcal{L}^{\mathit{fr}}(2n)$ be the set of all relative framed links in $(M, \partial M)$ considered up to ambient isotopy keeping $\partial M$ fixed, such that $L \cap \partial M = \partial L = \{x_i\}_1^{2n}$. Consider the submodule $S_{2,\infty}^{\mathit{sub}}(2n)$ of the free $R$-module $R\mathcal{L}^{\mathit{fr}}(2n)$ generated by the Kauffman bracket skein expressions. Then the {\bf relative Kauffman bracket skein module}, henceforth known as the relative skein module, of $M$ is the quotient: $$\mathcal{S}_{2,\infty}(M, \{x_i\}_1^{2n}; R, A) = \frac{R\mathcal{L}^{\mathit{fr}}(2n)}{ S_{2,\infty}^{\mathit{sub}}(2n)}.$$ 

\end{definition}

We will use the notation $\mathcal{S}_{2,\infty}(M,\{x_i\}_1^{2n})$ when $R = \mathbb{Z}[A^{\pm 1}]$. For thickened surface, the following theorem tells us the basis of  the (relative) Kauffman bracket skein module.

\begin{theorem}\label{ftimesi}\cite{smof3,fundamentals}
 
Let $\Sigma$ be an oriented surface in which each link is equipped with blackboard framing and let $I$ denote the unit interval $[0,1]$. Then $\mathcal{S}_{2,\infty}(\Sigma \times I;R,A)$ is a free $R$-module whose basis consists of the empty link $\varnothing$ and simple closed multicurves in $\Sigma$ that have no trivial components. This applies in particular to handlebodies, since $H_{n} = \Sigma_{0,n+1} \times I$, where $H_n$ is a handlebody of genus $n$ and $\Sigma_{g,b}$ denotes a genus $g$ surface with $b$ boundary components.
 
 \end{theorem}

 The following example discusses the skein module of the thickened annulus.

 \begin{example}\label{solidtoruskbsm}

$\mathcal S_{2,\infty} (\Sigma_{0,2} \times I; R, A)$ is free and infinitely generated by the curves $\{x^i\}_{i \geq 0}$, where $x$ denotes the homotopically nontrivial simple closed curve on the annulus and $x^0$ denotes the empty link $\varnothing$. Note that, $\mathcal S_{2,\infty} (S^1 \times D^2; R, A) \cong \mathcal S_{2,\infty} (\Sigma_{0,2} \times I; R, A)$. 
     
 \end{example}

 A result similar to Theorem \ref{ftimesi} also holds for relative skein modules.

 \begin{theorem}\label{rkbsmsib}\cite{fundamentals}

Let $\Sigma$ be an oriented surface, where $\partial \Sigma \neq \emptyset$, and let $\{x_i\}_1^{2n}$ be $2n$ oriented framed points centred at $\partial \Sigma \times \{\frac{1}{2}\}$. %Recall that elements of the RKBSM are relative framed links whose ends points are $x_i$. 
Then $\mathcal S_{2, \infty}(\Sigma \times I, \{x_i\}_1^{2n}; R, A)$ is a free $R$-module whose basis is composed of relative links\footnote{Relative links in $\Sigma$ are families of properly embedded arcs and closed curves with blackboard framing in $\Sigma \times \{\frac{1}{2}\}$.} in $\Sigma \times \{\frac{1}{2}\}$ without trivial components.

\end{theorem}

The skein module of a surface times an interval may be equipped with an algebra structure for which the multiplication operation is defined as follows.

\begin{definition}
 Consider two framed links $L_1$ and $L_2$ in $\Sigma \times I$. Define their product $\cdot$ by placing $L_1$ over $L_2$ in $\Sigma \times I$, that is, $L_1 \cdot L_2 = L_1 \sqcup L_2$ such that $L_{1} \subset \Sigma\times (\frac{1}{2}, 1)$ and $L_2 \subset \Sigma\times(0,\frac{1}{2})$. The empty link $\varnothing$ serves as the multiplicative identity. This multiplication endows the skein module of a thickened surface $\Sigma \times I$ with a natural algebra structure. The Kauffman bracket skein module equipped with this algebra structure is called the {\em Kauffman bracket skein algebra}.   \end{definition}

We denote the Kauffman bracket skein algebra, henceforth known simply as the skein algebra, by $\mathcal{S}^{\mathit{alg}}(\Sigma; R, A)$. This new notation emphasises the fact that the skein algebra depends on the surface and its product structure. For brevity, we use the notation $\mathcal{S}^{\mathit{alg}}(\Sigma)$ when $R = \mathbb{Z}[A^{\pm 1}]$.

We now state some properties of the skein module required for proving our main results. The following theorem determines how the Kauffman bracket skein module behaves under handle addition, thereby giving its presentation in terms of generators and relations. 

%\begin{theorem}\cite{fundamentals}\label{propkbsm1}

%\begin{enumerate}

    %\item Let $i : M \hookrightarrow N$ be an orientation preserving embedding of $3$-manifolds. This yields a homomorphism $i_* : \mathcal{S}_{2,\infty}(M;R,A) \longrightarrow
%\mathcal{S}_{2,\infty}(N;R,A)$ of skein modules. This correspondence leads to a functor from the
%category of $3$-manifolds and orientation preserving embeddings (up to ambient isotopy) to the category of $R$-modules with a specified invertible
%element $A \in R$.\label{embhom} \\

%\item \label{disjointunion}Let $M = M_1 \sqcup M_2$ be the disjoint union of oriented $3$-manifolds $M_1$ and $M_2$. Then 

%$$\mathcal S_{2,\infty}(M;R, A) \cong \mathcal S_{2,\infty}(M_1;R, A) \otimes_R \mathcal S_{2,\infty}(M_2;R, A).$$ \ \ 

%\item {\it (The Universal Coefficient Property)} \label{universalcoefkbsm}\noindent Let $R$ and $R'$ be commutative rings with unity and $r : R \longrightarrow R'$ be a homomorphism. %We can think of R′ as an R module. 
%Then the identity map on $\mathcal{L}^{\mathit{fr}}$
%induces the following isomorphism of $R'$ (and R) modules:
%$$\overline{r} : \mathcal{S}_{2,\infty}(M;R,A) \otimes_R  R' \longrightarrow \mathcal{S}_{2,\infty}(M;R', r(A)).$$

%\end{enumerate}

%\end{theorem}

%The following theorem determines how the Kauffman bracket skein module behaves under handle addition, thereby giving its presentation in terms of generators and relations. 
    \begin{theorem}\cite{fundamentals,kbsmlens}\label{propkbsm2}

    \begin{enumerate}

    \item  If $N$ is obtained from $M$ by adding a $3$-handle to $M$ and $i : M \hookrightarrow N$ is the associated embedding, then the induced homomorphism 
$i_* : \mathcal{S}_{2,\infty}(M;R,A) \longrightarrow \mathcal{S}_{2,\infty}(N;R,A)$ is an isomorphism.\label{3hand} \\

\item (Handle Sliding Lemma)\label{handleslidinglemma} Let $(M,\partial M)$ be a $3$-manifold with boundary and $\gamma$ be a
simple closed curve on $\partial M$. Additionally, let $N = M_{\gamma}$ be the $3$-manifold
obtained from $M$ by adding a $2$-handle along $\gamma$ and $i : M \hookrightarrow N$ be the associated embedding. Then the induced homomorphism
$i_* : \mathcal{S}_{2,\infty}(M;R,A) \longrightarrow \mathcal{S}_{2,\infty}(N;R,A)$ is an epimorphism.\label{2hand} Furthermore, the kernel of $i_*$ is generated by the relations yielded by $2$-handle slidings. In particular, if $\mathcal{L}^{\mathit{fr}}_{\mathit{gen}}$
is a set of framed links in $M$ that generates $\mathcal{S}_{2,\infty}(M;R,A)$,
then $\mathcal{S}_{2,\infty}(N;R,A) \cong \mathcal{S}_{2,\infty}(M;R,A)/\mathcal J$, where $\mathcal J$ is the submodule of
$\mathcal{S}_{2,\infty}(M;R,A)$ generated by the expressions $L - sl_{\gamma}(L)$.  Here $L \in \mathcal{L}^{\mathit{fr}}_{\mathit{gen}}$ and $sl_{\gamma}(L)$ is obtained from $L$ by sliding it along $\gamma$. %(that is, we perform a $2$-handle sliding).
\end{enumerate}

\end{theorem}

The following is about how to get the submodule $\mathcal J$ consisting of handle sliding relations via its relative skein module. For instance, consider the two marked points $a$ and $b$, such that they lie on the simple closed curve $\mu$ in $\partial H_2$ and they divide the curve $\mu$ into two arcs $\alpha$ and $\beta$ (see Figure \ref{gamma12}). \\

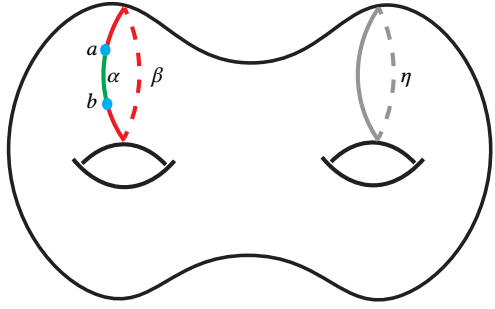
\begin{figure}[ht]
    \centering
    \begin{overpic}[unit=1mm, scale = 0.17]{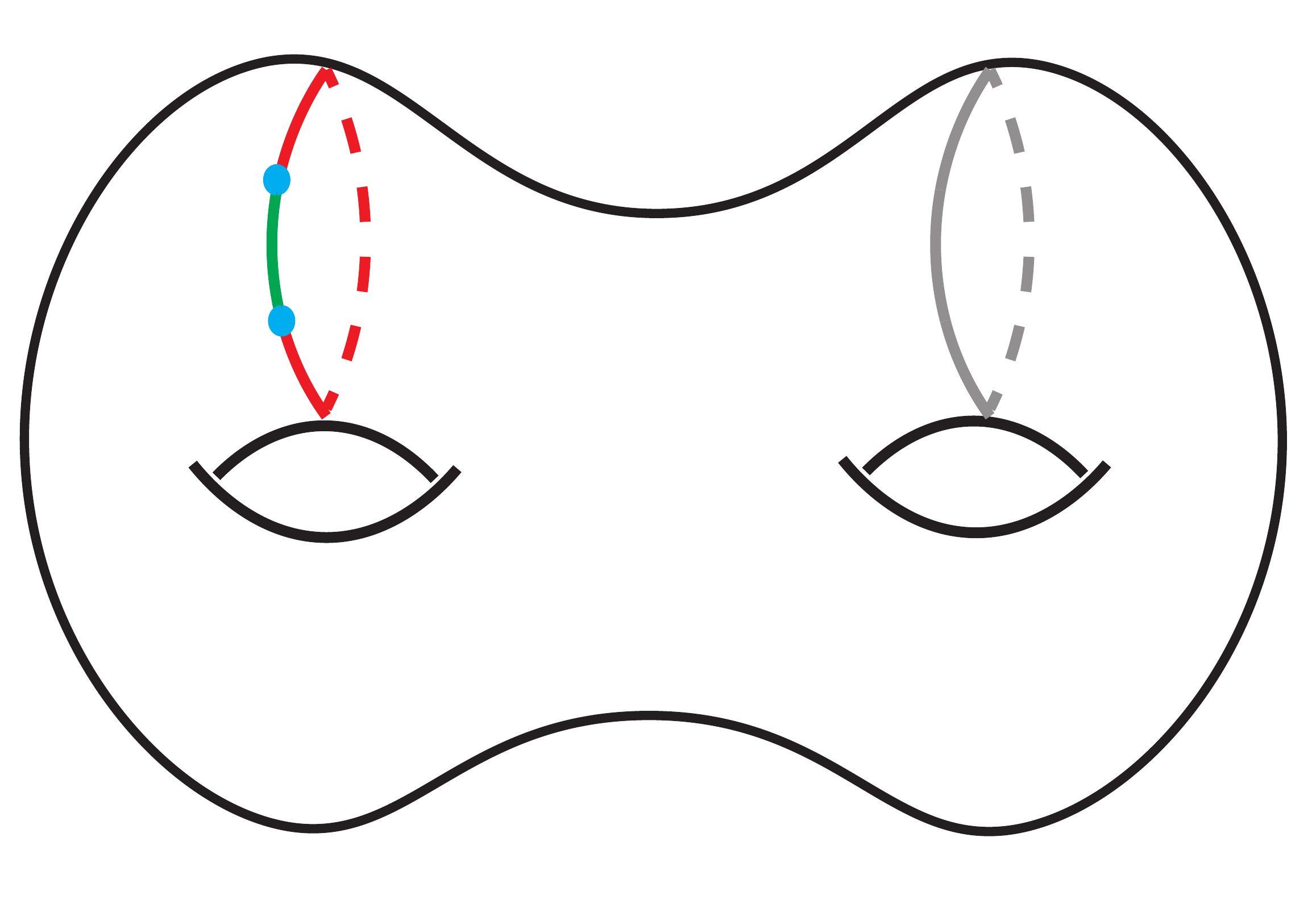}
\put(11.5,29){\tiny{$b$}}
\put(11.5,36){\tiny{$a$}}
%\put(14,6){\tiny $ \beta$}
\put(14.25,32.5){\tiny $\alpha$}
\put(20,32.5){\tiny$\beta$}
\put(53,32.5){\tiny$\eta$}
\end{overpic}
\vspace{-5mm}
    \caption{Marked points $a$ and $b$ on the simple closed curve $\mu$ in $\partial H_2$ that divide it into curves $\alpha$ and $\beta$.}
    \label{gamma12}
\end{figure}

Consider any relative curve $L$ in $(H_2; a, b)$. Now, handle slidings in $(H_2)_{\mu}$ take place locally in the neighbourhood of the curve $\beta$. Consider fixed tangents at the points $a$ and $b$ and let the relative curve $L$ approach these points along the tangents. For every relative curve $L$, handle sliding in $(H_2)_{\mu}$ replaces the curve $L\cup \alpha$ with the curve $L \cup \beta$. This gives the handle sliding relation, $L \cup \alpha \equiv L \cup \beta$. By introducing the $\mathbb Z[A^{\pm 1}]$-linear homomorphism $\omega: \mathcal S_{2,\infty}(H_2; a,b) \longrightarrow \mathcal S_{2,\infty}(H_2)$, defined by $\omega (L) = L \cup \alpha - L \cup \beta$, we see that $\omega(\mathcal S_{2,\infty}(H_2; a,b)) = \mathcal J$. Hence, the image of any basis of $\mathcal S_{2,\infty}(H_2; a,b)$ generates $\mathcal J$. See Figure \ref{omegaalpha} for a visual explanation. We note that this method of describing $2$-handle sliding relations was pioneered by Bullock and Lo Faro in \cite{knotext}. See also \cite{blp}.

\begin{figure}[ht]
$\vcenter{\hbox{\begin{overpic}[scale=.13]{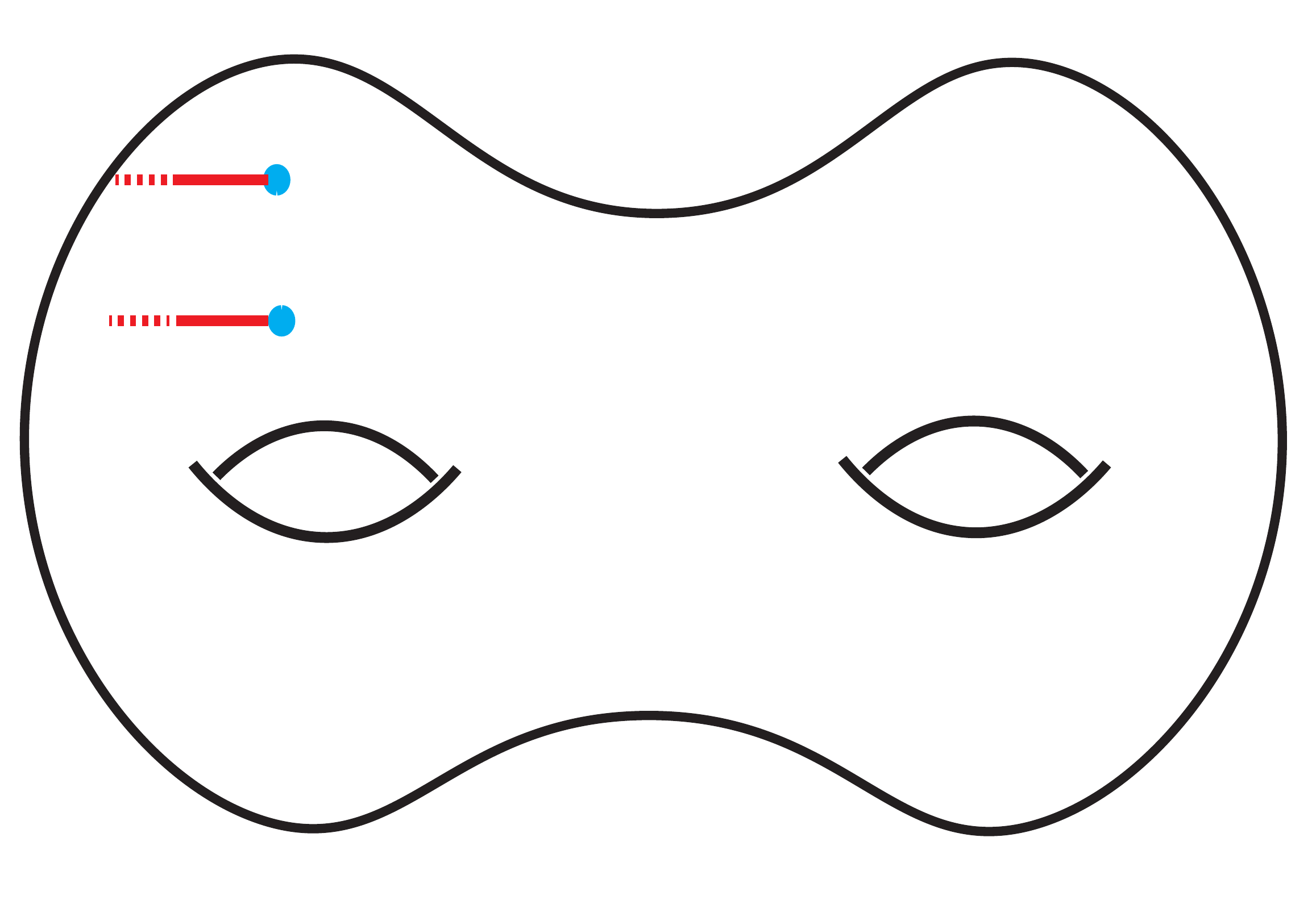}
\put(66, -5){$\alpha$}
%\put(36, 29){$F_n$}
\end{overpic}}}   \xrightarrow{\omega}
\vcenter{\hbox{\begin{overpic}[scale=.13]{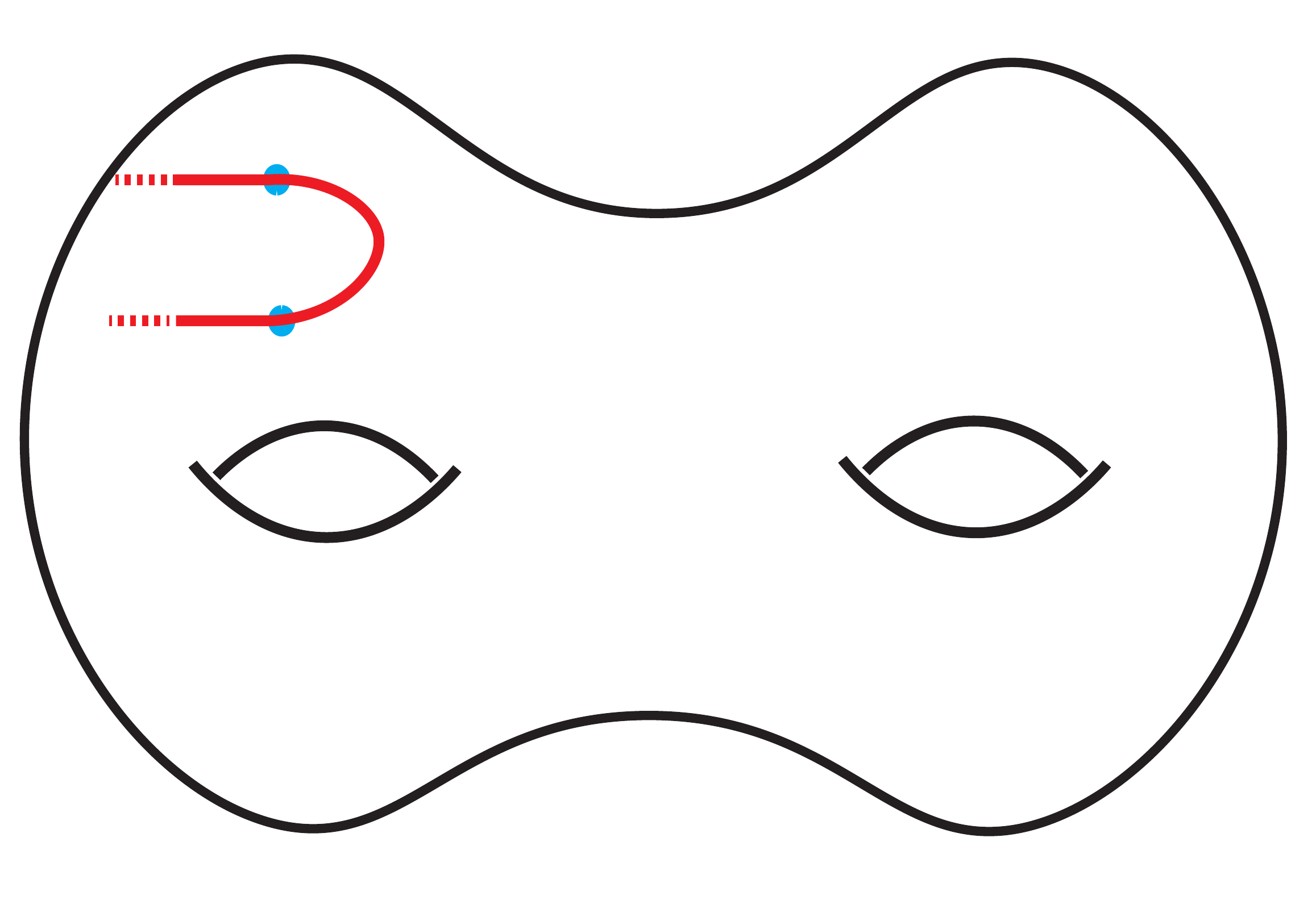}
\put(56, -5){$\alpha \cup \beta_2$}
%\put(36, 29){$F_n$}
\end{overpic}}} 
 -
\vcenter{\hbox{\begin{overpic}[scale=.13]{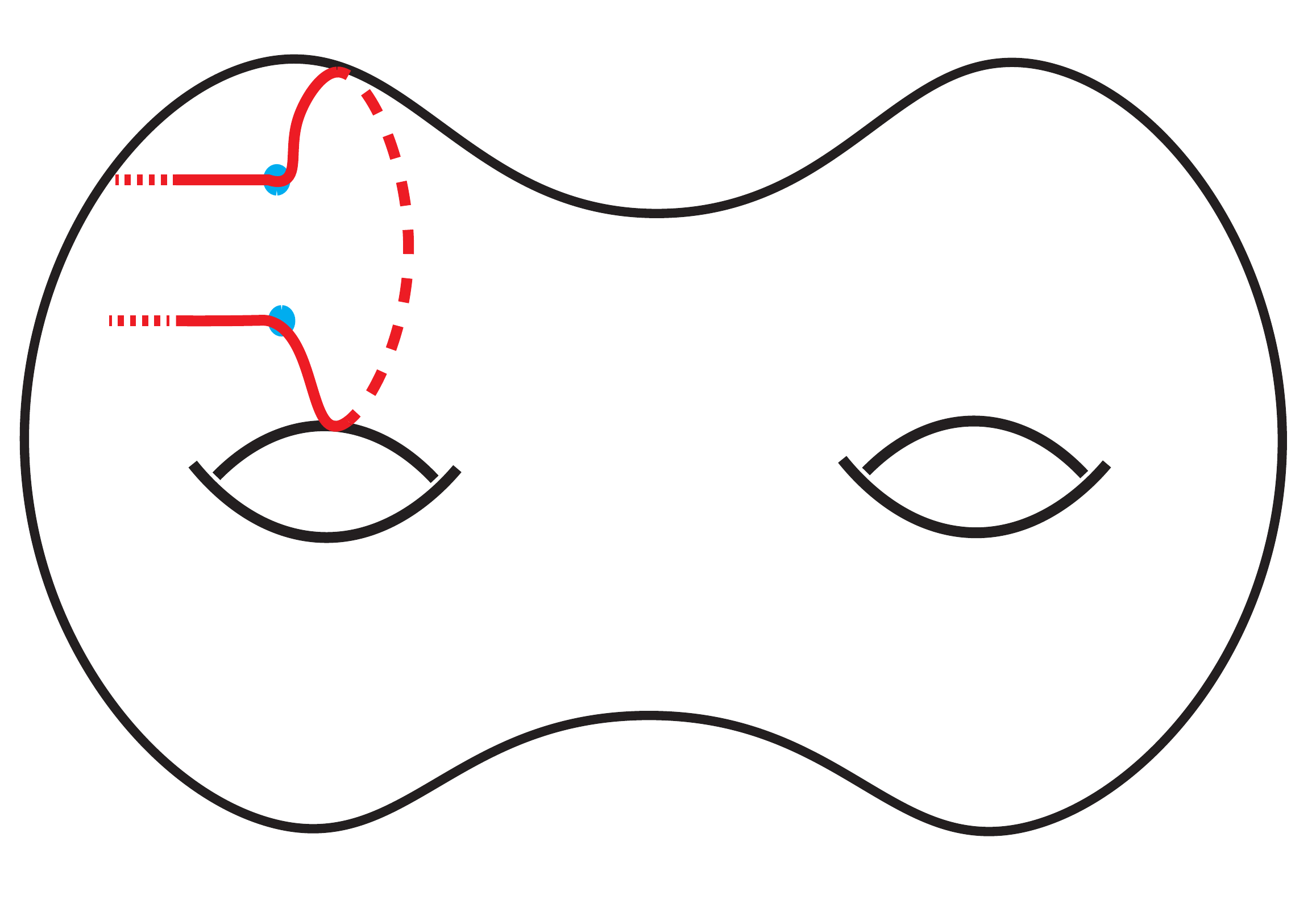}
\put(56, -5){$\alpha \cup \beta_1$}
%\put(5.5, 33){\footnotesize{$n-1$}}
%\put(92, 33){\footnotesize{$n-1$}}
%\put(48.5, 25){\footnotesize{$n-2$}}
%\put(36, 29){$F_n$}
\end{overpic}}}$
\vspace*{3mm}
\caption{Illustration of $\omega(\alpha)$.}
\label{omegaalpha}
\end{figure}

\section{Heegaard skein complex}
Let us denote by $H_{g_0}$ the standard gunes $g_0$ handlebody, and let $\vec\mu=\{\mu_{1},\mu_{2},\dots,\mu_{g}\}$ be an ordered set of $g<g_0$ simple closed curves $\mu_i:[-1,1]\rightarrow \partial(H_{g_0})$ so that they are mutually disjoint, non-separating and the boundary surface remains connected after removing them. These information defines a three dimensional manifold obtained from a genus $g_0$ handlebody by attaching $g$ $2$-handles along the $g$ attaching curves $\mu_i$ for $i=1,\dots g$.  Let $B_i$ be a tubular neighborhood of $\mu_i$ with a local coordinate $B_i:[-1,1]\times (-\infty,+\infty)\rightarrow \partial(H_{g_0})$ with $\mu_i=B_i([-1,1]\times \{1\})$. Next, for $n\geq 1$, we create countably many parallel copies of $\mu_i$ denoted by $\mu_{i,n}$ such that points on $\mu_{i,n}$ have second coordinate $n$ (i.e. $\mu_i$ is $\mu_{i,1}$). On each $\mu_{i,n}$, we put a pair of framed points $a_{i,n}$ and $b_{i,n}$ at position $(-\epsilon,n)$ and $(\epsilon,n)$ respectively, and the framing is chosen such that it is pointing to the right when walking along the curves. The two framed points separates the curve into two parts, we write $\alpha_{i,n}$ for the part starting from $a_{i,n}$ to $b_{i,n}$ following the orientation and write $\beta_{i,n}$ for the other part.
Let us denote the ordered set of curves $\mu$ together with the additional information in the previous paragraph by $U=(\mu,B_i,\mu_{i,n},a_{i,n},b_{i,n})$.  write $u_i$ for $\mu_i$ with chosen framed points $a_{i}$ and $b_i$.

With these setting, we define the following key definition of the chain groups:

\begin{definition}

Let $X_{n}(H_{g_0},U)$ be the set of relative framed links with $2n$ endpoints that are attached to $n$ pairs of framed points on certain $B_i$s and on each $B_i$, we require that the pairs of framed points with smaller second coordinate are occupied first. Define the $n$th chain group $C_{n}(H_{g_0},U)=\mathbb{Z}[A,A^{-1}]X_{n}(H_{g_0},U)/KBSR$, i.e. the quotient module of the free $\mathbb{Z}[A,A^{-1}]$ module with basis elements in $X_{n}(H_{g_0},U)$ by the submodule generated by the Kauffman bracket skein relations as shown in Figure \ref{skeintriple}.
    
\end{definition}
To define the boundary map, we need some further notations. For each $L\in X_{n}(H_{g_0},U)$, write $n_i(L)$ to be the number of pairs of framed points occupied by $L$ in $B_i$(usually just write $n_i$ when it is clear from context). For later convenience, we also denote the endpoints of $L$ by $\bar{a}_{i,n}$ and $\bar{b}_{i,n}$ so that $\bar{a}_{i,n}$ and $\bar{b}_{i,n}$ are attached to $a_{i,n}$ and $b_{i,n}$ respectively.

\begin{definition}

We define a map $$W_{i,j}:C_{n}(H_{g_0},\mu)\rightarrow C_{n-1}(H_{g_0},\mu)$$ such that, for each $L\in X_{n}(H_{g_0},U)$, $$W_{i,j}(L)=L\cup \alpha_{i,j}-L\cup \beta_{i,j}\footnote{Sometimes, when strengthening the place where the union take place, we write $L\bigcup\limits_{(\bar{a}_{i,n},\bar{b}_{i,n})} \alpha_{i,j}$.}.$$
Since the Kauffman bracket skein relations are local, $W_{i,j}$ is well-defined for each $L$ and we extend linearly.    
\end{definition}

\begin{remark}

In this definition, we consider $\alpha_{i,j}$ and $\beta_{i,j}$  to be part framed line segments(after pushing inside the handlebody) equipped with the blackboard framing on $\partial(H_{g_0})$ coincide with the framing of the framed points $a_{i,j}$ and $b_{i,j}$ on them. After taking the union, we push $L\cup \alpha_{i,j}$ and $L\cup \beta_{i,j}$ slightly into the handlebody $H_{g_0}$. Notice that $W_{i,j}(L)$ may not belong to $C_{n-1}(H_{g_0},U)$ since some pair of framed points on $B_i$ with smaller second coordinate are no longer occupied. We will fix this issue in the following definition.
    
\end{remark}

Now let us define a shifting operator $S$ in the following way: for each framed link $\tilde{L}$ and suppose that there are certain gaps in the sequence of the occupied pairs of framed points. Say $(a_{i,k_i},b_{i,k_i})$ is a pair of framed points that are gaped in the sequence of the occupied pairs of framed points on $B_i$, then we shift $\tilde{L}$ by moving the pair of endpoints $(\bar{a}_{i+1,k_{i}+1},\bar{b}_{i+1,k_{i}+1})$ on $\partial(H_{g_0})$ along with $\tilde{L}$ to $(a_{i,k_i},b_{i,k_i})$ in the obvious way inside a small disk enclosing four framed points $a_{i,k_i},b_{i,k_i},a_{i+1,k_{i}+1},b_{i+1,k_{i}+1}$. The effect of $S$ is to shift $\tilde{L}$ at each gap so that all of them are removed. Therefore, $S(\tilde{L})$ is an element in certain $C_{n}(H_{g_0},U)$.

\begin{definition}

Let $\partial_{i,j}:C_{n}(H_{g_0},\mu)\rightarrow C_{n-1}(H_{g_0},U)$ be the homorphism that sends each framed link $L$ to $S\circ W_{i,j}(L)$ and extend linearly. We define the boundary map $$\partial_n:C_{n}(H_{g_0},U)\rightarrow C_{n-1}(H_{g_0},U)$$ by $$\partial_n(L)=\sum\limits_{i=1}^{g}(-1)^{\sum\limits_{s=1}^{i-1}n_s}\sum\limits_{j=1}^{n_i}(-1)^j\partial_{i,j}(L).$$
The chain complex is called the Heegaard skein complex $(C_{n}(H_{g_0},U),\partial_n)$ and its homology is called the Heegaard skein homology.
    
\end{definition}

It remains to check that $\partial_{n-1}\circ \partial_{n}=0$.

\begin{proposition}

For any given data $(H_{g_0},\mu)$, $(C_{n}(H_{g_0},U),\partial_n)$ forms a chain complex. 
    
\end{proposition}

\begin{proof}

We prove $\partial_{n-1}\circ \partial_{n}=0$ by showing that the chain complex  $(C_{n}(H_{g_0},U),\partial_n)$ indeed comes from a presimplicial module. Let $p(i,j)=\sum\limits_{k=1}^{i-1}n_k+j$ and $p^{-1}(i)=(r_i,s_i)$. Write $\partial_{i,j}$ as $d_{p(i,j)}$, we only need to show that for all $i<j$, $d_{i}\circ d_j=d_{j-1}\circ d_i$ is true. We check it directly by the following:

\begin{enumerate}
    \item Suppose $p^{-1}(i)=(r_i,s_i)$ and $p^{-1}(j)=(r_j,s_j)$ with $r_i\neq r_j$:
    \begin{eqnarray*}
        &&d_i\circ d_j(L) \\&=& d_i(L\cup\alpha_{r_j}^{s_j}-L\cup\beta_{r_j}^{s_j})\\
        &=&S[(L\cup\alpha_{r_j}^{s_j})\cup\alpha_{r_i}^{s_i}-(L\cup\alpha_{r_j}^{s_j})\cup\beta_{r_i}^{s_i}-(L\cup\beta_{r_j}^{s_j})\cup\alpha_{r_i}^{s_i}+(L\cup\beta_{r_j}^{s_j})\cup\beta_{r_i}^{s_i}]\\  
        &&d_{j-1}\circ d_i(L) \\&=& d_{j-1}(L\cup\alpha_{r_i}^{s_i}-L\cup\beta_{r_i}^{s_i})\\   &=&S[(L\cup\alpha_{r_i}^{s_i})\cup\alpha_{r_j}^{s_j}-(L\cup\alpha_{r_i}^{s_i})\cup\beta_{r_j}^{s_j}-(L\cup\beta_{r_i}^{s_i})\cup\alpha_{r_j}^{s_j}+(L\cup\beta_{r_i}^{s_i})\cup\beta_{r_j}^{s_j}]  
    \end{eqnarray*}
    \item Suppose $r_i= r_j$:
    \begin{eqnarray*}
        &&d_i\circ d_j(L) \\&=& d_i(L\cup\alpha_{r_j}^{s_j}-L\cup\beta_{r_j}^{s_j})\\  &=&S[(L\cup\alpha_{r_j}^{s_j})\cup\alpha_{r_i}^{s_i}-(L\cup\alpha_{r_j}^{s_j})\cup\beta_{r_i}^{s_i}-(L\cup\beta_{r_j}^{s_j})\cup\alpha_{r_i}^{s_i}+(L\cup\beta_{r_j}^{s_j})\cup\beta_{r_i}^{s_i}]\\ 
        &&d_{j-1}\circ d_i(L)\\
        &=& d_{j-1}(L\cup\alpha_{r_i}^{s_i}-L\cup\beta_{r_i}^{s_i})\\
        &=& S[(L\cup\alpha_{r_i}^{s_i})\cup\alpha_{r_j}^{s_j-1}-(L\cup\alpha_{r_i}^{s_i})\cup\beta_{r_j}^{s_j-1}-(L\cup\beta_{r_i}^{s_i})\cup\alpha_{r_j}^{s_j-1}+(L\cup\beta_{r_i}^{s_i})\cup\beta_{r_j}^{s_j-1}]  
    \end{eqnarray*}     
\end{enumerate}
 In this proof, all the unions with $\alpha_{r_i}^{*}$($\beta_{r_i}^{*}$) are over the pair of endpoints $(\bar{a}_{r_i}^{s_i},\bar{b}_{r_i}^{s_i})$ and  all the unions with $\alpha_{r_j}^{*}$($\beta_{r_j}^{*}$) are over the pair of endpoints $(\bar{a}_{r_j}^{s_j},\bar{b}_{r_j}^{s_j})$. Notice that the order of taking unions does not matter since they are not intersecting at all. For $(1)$, clearly they are the same; For $(2)$, though they are not identically the same, they are ambient isotopic to each other. Combine case $(1)$ and $(2)$, we complete the proof.
\end{proof}

In our definition, the chain complex depends on several choices: the exact position of framed points, the orientation of curves in $\mu$ and the ordering of the curves in $\mu$. We are going to show that the chain complexes are chain homotopy equivalent under different choices.

\begin{proposition}\label{ai}

Given $(H_{g_0},U)$ and $(H_{g_0},U')$, and $U'$ only differs from $U$ on the $i
$th curve such that $\mu_{i}'$ is the same as $\mu_i$, but $b_i$ is changed to another position $b_{i}'$ on $\mu_i$, then the two chain complexes $(C_{n}(H_{g_0},U),\partial_n)$ and $(C_{n}(H_{g_0},U'),\partial_n')$ are chain homotopy equivalent.
    
\end{proposition}

\begin{proof}

There is a homeomorphism from $\partial(H_{g_0})$ to itself that fixes $a_i$ that moves $b_i$ along $\mu_i$ that fix the boundary of a tubular neighborhood of $\mu_i$. This homeomorphism extends to a homeomorphism from $H_{g_0}$ to itself that only moves the points in a collar neighborhood of $\partial(H_{g_0})$. This homeomprphism induces a chain isomorphism.

\end{proof}

\begin{proposition}\label{ordering}

Given $(H_{g_0},U)$ and $(H_{g_0},U')$, and $U'$ only differs from $U$ by a change of ordering, then the two chain complexes$(C_{n}(H_{g_0},U),\partial_n)$ and $(C_{n}(H_{g_0},U'),\partial_n')$ are chain homotopy equivalent.
    
\end{proposition}

\begin{proof}
Without the lose of generality, we only need to consider the case where we switch $\mu_1$ and $\mu_2$. For any $L\in X_{n}(H_{g_0},U)$, we define a map $$f_n:(C_{n}(H_{g_0},U),\partial_n)\rightarrow (C_{n}(H_{g_0},U'),\partial_n')$$ 
by $$f_n(L)=(-1)^{n_1n_2}\overline{L}$$
where $\overline{L}$ is the same map as $L$ but considered as an element in $C_{n}(H_{g_0},U')$.

   {\small 
   \begin{eqnarray*}
        &&\partial_n'\circ f_n(L)=(-1)^{n_1n_2}\partial_n'(\overline{L})\\
        &=&(-1)^{n_1n_2}[\sum\limits_{t=1}^{n_2}(-1)^t\overline{\partial_{2,t}(L)}+(-1)^{n_2}\sum\limits_{t=1}^{n_1}(-1)^t\overline{\partial_{1,t}(L)}+(-1)^{n_1+n_2}\sum\limits_{i=3}^{g}(-1)^{\sum\limits_{s=1}^{i-1}n_s}\sum\limits_{t=1}^{n_i}(-1)^t\overline{\partial_{i,t}(L)}]\\
        &=&(-1)^{(n_1+1)n_2}\sum\limits_{t=1}^{n_1}(-1)^t\overline{\partial_{1,t}(L)}+(-1)^{n_1n_2}\sum\limits_{t=1}^{n_2}(-1)^t\overline{\partial_{2,t}(L)}+(-1)^{n_1+n_2+n_1n_2}\sum\limits_{i=3}^{g}(-1)^{\sum\limits_{s=1}^{i-1}n_s}\sum\limits_{t=1}^{n_i}(-1)^t\overline{\partial_{i,t}(L)}]\\ 
        &&f_{n-1}\circ \partial_n(L)=f_{n-1}(\sum\limits_{i=1}^{g}(-1)^{\sum\limits_{s=1}^{i-1}n_s}\sum\limits_{j=1}^{n_i}(-1)^t\partial_{i,t}(L))\\
        &=&f_{n-1}(\sum\limits_{t=1}^{n_1}(-1)^t\partial_{1,t}(L)+(-1)^{n_1}\sum\limits_{t=1}^{n_2}(-1)^t\partial_{2,t}(L)+(-1)^{n_1+n_2}\sum\limits_{i=3}^{g}(-1)^{\sum\limits_{s=1}^{i-1}n_s}\sum\limits_{j=1}^{n_i}(-1)^t\partial_{i,t}(L))\\
        &=&(-1)^{(n_1-1)n_2}\sum\limits_{t=1}^{n_1}(-1)^t\overline{\partial_{1,t}(L)}+(-1)^{n_1}(-1)^{n_1(n_2-1)}\sum\limits_{t=1}^{n_2}(-1)^t\overline{\partial_{2,t}(L)}\\
        &&+(-1)^{n_1+n_2}(-1)^{n_1n_2}\sum\limits_{i=3}^{g}(-1)^{\sum\limits_{s=1}^{i-1}n_s}\sum\limits_{t=1}^{n_i}(-1)^t\overline{\partial_{i,t}(L)}\\     
        &=&(-1)^{(n_1-1)n_2}\sum\limits_{t=1}^{n_1}(-1)^t\overline{\partial_{1,t}(L)}+(-1)^{n_1n_2}\sum\limits_{t=1}^{n_2}(-1)^t\overline{\partial_{2,t}(L)}+(-1)^{n_1+n_2+n_1n_2}\sum\limits_{i=3}^{g}(-1)^{\sum\limits_{s=1}^{i-1}n_s}\sum\limits_{t=1}^{n_i}(-1)^t\overline{\partial_{i,t}(L)}    
    \end{eqnarray*}
    }
    Compare the above formula, we have $f_n$ is a chain map. Similarly, one can define $$g_n:(C_{n}(H_{g_0},U'),\partial_n')\rightarrow (C_{n}(H_{g_0},U),\partial_n)$$ 
by $$g_n(L)=(-1)^{n_1n_2}\overline{L}$$
which is clearly the inverse of $f_n$. Thus, we finish the proof.
\end{proof}

\begin{proposition}\label{orientation}

Given $(H_{g_0},U)$ and $(H_{g_0},U')$, and $U'$ only differs from $U$ by a change orientation on one of the $\mu_is$, then the two chain complexes$(C_{n}(H_{g_0},U),\partial_n)$ and $(C_{n}(H_{g_0},U'),\partial_n')$ are chain homotopy equivalent.
    
\end{proposition}

\begin{proof}
Based on our definition, if we change the orientation of one curve $\mu_i$, the position of the pair of framed points on $\mu_i$ will switch and the framing will point to the opposite direction. For any $L\in X_{n}(H_{g_0},U)$, we define a map $$f_n:(C_{n}(H_{g_0},U),\partial_n)\rightarrow (C_{n}(H_{g_0},U'),\partial_n')$$ 
by $$f_n(L)=(-1)^{\frac{n(n-1)}{2}}\overline{L}$$
where $\tilde{L}$ is obtained from $L$ in the following manner:
\begin{enumerate}
    \item $n_i=0$, $\overline{L}$ is the same map as $L$ but considered as an element in $C_{n}(H_{g_0},U')$.
    \item $n_i>0$, we can find a small disk $D$ on $\partial(H_{g
    _0})$ containing $\{a_{i,1},\dots,a_{i,n_i},b_{i,1},\dots,b_{i,n_i},$\\$a_{i,1}',\dots,a_{i,n_i}',b_{i,1}',\dots,b_{i,n_i}'\}$. Inside this disk, we shift our link $L$ so that the pair of endpoints of $L$ originally connected to $(a_{i,j},b_{i,j})$ is now connected to $(a_{i,n_i-j+1}',b_{i,n_i-j+1}')$. Notice that, at this stage, the framing of the shifted $L$ does not match with the framing of the framed points $a_{i,j}'$ and $b_{i,j}'$, so we further twist each end of $L$ connecting to $a_{i,j}'$ by a positive half twist and that of $L$ connecting to $b_{i,j}'$ by a negative half twist. We denoted it by $\overline{L}$ considered as an element in $C_{n}(H_{g_0},U')$.
\end{enumerate}
Next, without the loss of generality, we check $f_n$ is a chain map for $i=1$:
       \begin{eqnarray*}
        &&\partial_n'\circ f_n(L)=\partial_n'(\overline{L})\\
        &=&[\sum\limits_{t=1}^{n_1}(-1)^t\partial_{1,t}(\overline{L})+(-1)^{n_1}\sum\limits_{i=2}^{g}(-1)^{\sum\limits_{s=1}^{i-1}n_s}\sum\limits_{t=1}^{n_i}(-1)^t\partial_{i,t}\overline{L}]\\
        &&f_{n-1}\circ \partial_n(L)=f_{n-1}(\sum\limits_{i=1}^{g}(-1)^{\sum\limits_{s=1}^{i-1}n_s}\sum\limits_{t=1}^{n_i}(-1)^t\partial_{i,t}(L))\\
        &=&f_{n-1}(\sum\limits_{t=1}^{n_1}(-1)^t\partial_{1,t}(L)+(-1)^{n_1}\sum\limits_{i=2}^{g}(-1)^{\sum\limits_{s=1}^{i-1}n_s}\sum\limits_{t=1}^{n_i}(-1)^t\partial_{i,t}(L))\\
        &=&\sum\limits_{t=1}^{n_1}(-1)^t\overline{\partial_{1,t}(L)}+(-1)^{n_1}\sum\limits_{i=2}^{g}(-1)^{\sum\limits_{s=1}^{i-1}n_s}\sum\limits_{t=1}^{n_i}(-1)^t\overline{\partial_{i,t}(L)}  
    \end{eqnarray*}
    For $i>1$, $\overline{\partial_{i,t}(L)}=\partial_{i,t}\overline{(L)}$ since the gluing operation with $\alpha_{i,j}$($\beta_{i,j}$) happens far away from the operation of $f_n$. For $i=1$, we compare the following:
    \begin{eqnarray*}
        &&(-1)^t\overline{\partial_{1,t}(L)}\\
        &=&(-1)^t(-1)^{\frac{(n_1-1)(n_1-2)}{2}}(\overline{S(L\bigcup_{(\bar{a}_{1,t},\bar{b}_{1,t})}\alpha_{1,t}})-\overline{S(L\bigcup_{(\bar{a}_{1,t},\bar{b}_{1,t})}\beta_{1,t}}))\\
        &&(-1)^{n_1-t+1}\partial_{1,n_1-t+1}((-1)^{\frac{n_1(n_1-1)}{2}}\overline{L})\\
        &=&(-1)^{n_1-t+1}(-1)^{\frac{n_1(n_1-1)}{2}}(S(\overline{L}\bigcup_{(\bar{a}_{1,t},\bar{b}_{1,t})}\alpha_{1,n_1-t+1}')-S(\overline{L}\bigcup_{(\bar{a}_{1,t},\bar{b}_{1,t})}\beta_{1,n_1-t+1}'))
    \end{eqnarray*}
    Notice that the coefficients are exactly the same. By our definitions, the two half twists will bring $\alpha'_{1,n_1-t+1}$ back to $\alpha_{1,n_1-t+1}$. Then, by an obvious ambient isotopy only move points in a collar neighborhood of the disk $D$, we have $\overline{S(L\bigcup\limits_{(\bar{a}_{1,t},\bar{b}_{1,t})}\alpha_{1,t})}=S(\overline{L}\bigcup\limits_{(\bar{a}_{1,t},\bar{b}_{1,t})}\alpha_{1,n_1-t+1}')$ and $\overline{S(L\bigcup\limits_{(\bar{a}_{1,t},\bar{b}_{1,t})}\beta_{1,t})}=S(\overline{L}\bigcup\limits_{(\bar{a}_{1,t},\bar{b}_{1,t})}\beta_{1,n_1-t+1}')$. This homomorphism is clearly an isomorphism.
\end{proof}

\begin{remark}
    By the definition of the Heegaard skein complex above and Theorem \ref{propkbsm2}(and the discussion there blow), we have the zeroth homology group $H_{0}(H_{g_0},U)$ is exactly the Kauffman bracket skein module of the three manifold $M$ obtained from $H_{g_0}$ with $2$-handles attached based on the information in $U$.
\end{remark}

\section{the invariance under 2-handle sliding moves}
In this section, we will construct a chain isomorphism between the chain complexes of $(H_{g_0},U)$ and $(H_{g_0},U')$ respectively, where $U'$ differs from $U$ by a $2$-handle sliding. In the following context, when we talk about the curves, when it is clear, we do not distinguish the map and its image curve.

Suppose $(H_{g_0},U')$ is obtained from $(H_{g_0},U)$ by sliding $\mu_i$ over $\mu_j$, and we denote the new curve by $\mu_{i}'$, the framed points by $a_{i}'$ and $b_{i}'$ and the two curves $\alpha_{i}'$ and $\beta_{i}'$. Let us consider the following setup:
\begin{enumerate}
    \item By connectedness of $\partial(H_{g'})-\cup_i \mu_i$, we can move $\alpha_i$ and $\alpha_j$ to a small disk, and a local picture is illustrated as in Figure \ref{n2hs}(ignoring the red arcs). Notice that, by Proposition \ref{orientation}, the local picture is always achievable by changing orientations if necessary. Let us choose $\mu_{i}'$ so that $\beta_i \subset \mu_{i}'$. We choose the framed points on $\mu_{i}'$ by $a_{i}'=a_i$ and $b_{i}'=b_i$, so we have $\beta_{i}'=\beta_i$ and $\alpha_{i}'$ is a parallel copy of $b_{ij}\cup \beta_j\cup a_{ij}$ pushed off on $\partial(H_{g'})$($\alpha_{i}'$ is to the right of $\beta_j$ following orientation). The orientation of $\mu_{i}'$ is chosen so that it coincides with the orientation on $\beta_{i}'=\beta_i$. Also, we setup its tubular neighborhood $B_{i}'$ and countably parallel copies as before.\footnote{Here we make such kind of setup is to make sure the chain map defined later commutes with differentials.}

%\item For each $L\in X_{n}(H_{g_0},U)$ with $n_i$ pairs of endpoints attached on $B_i$ and $n_j$ pairs of endpoints attached on $B_j$, we consider Figure \ref{n2hs}. For $0\leq k\leq n_i$, let us give some notations for certain elements $L_{k}$ in $X_{n}(H_{g_0},U')$: 
    %$L_{k}=L\cup b_{ij}^k \cup a_{ij}^k$ for $1\leq k\leq n_i$ and $L_{0}$ that is the same embedding as $L$ but viewed as a relative framed link in the three manifold  represented by $(H_{g_0},U')$.

    \item For each $L\in X_{n}(H_{g_0},U)$ with $n_i$ pairs of endpoints attached on $B_i$ and $n_j$ pairs of endpoints attached on $B_j$, we consider Figure \ref{n2hs}. We need some notation for certain framed links in $C_{n}(H_{g_0},U')$: write $\vec n_i=\{1,2,\dots,n_i\}$ as an ordered set, let $T=\{t_1,\dots,t_{|T|}\}$ be an ordered subset of $\vec n_i$, define $L_{T}\footnote{ $L_{\emptyset}$ is the same embedding as $L$ but viewed as a relative framed link in the three manifold represented by $(H_{g_0},U')$. }=S(L\bigcup\limits_{t\in T}(b_{i,j}^t\cup a_{i,j}^t))$.
    
\end{enumerate}

\begin{figure}[H]
    \centering
    \includegraphics[width=0.9\linewidth]{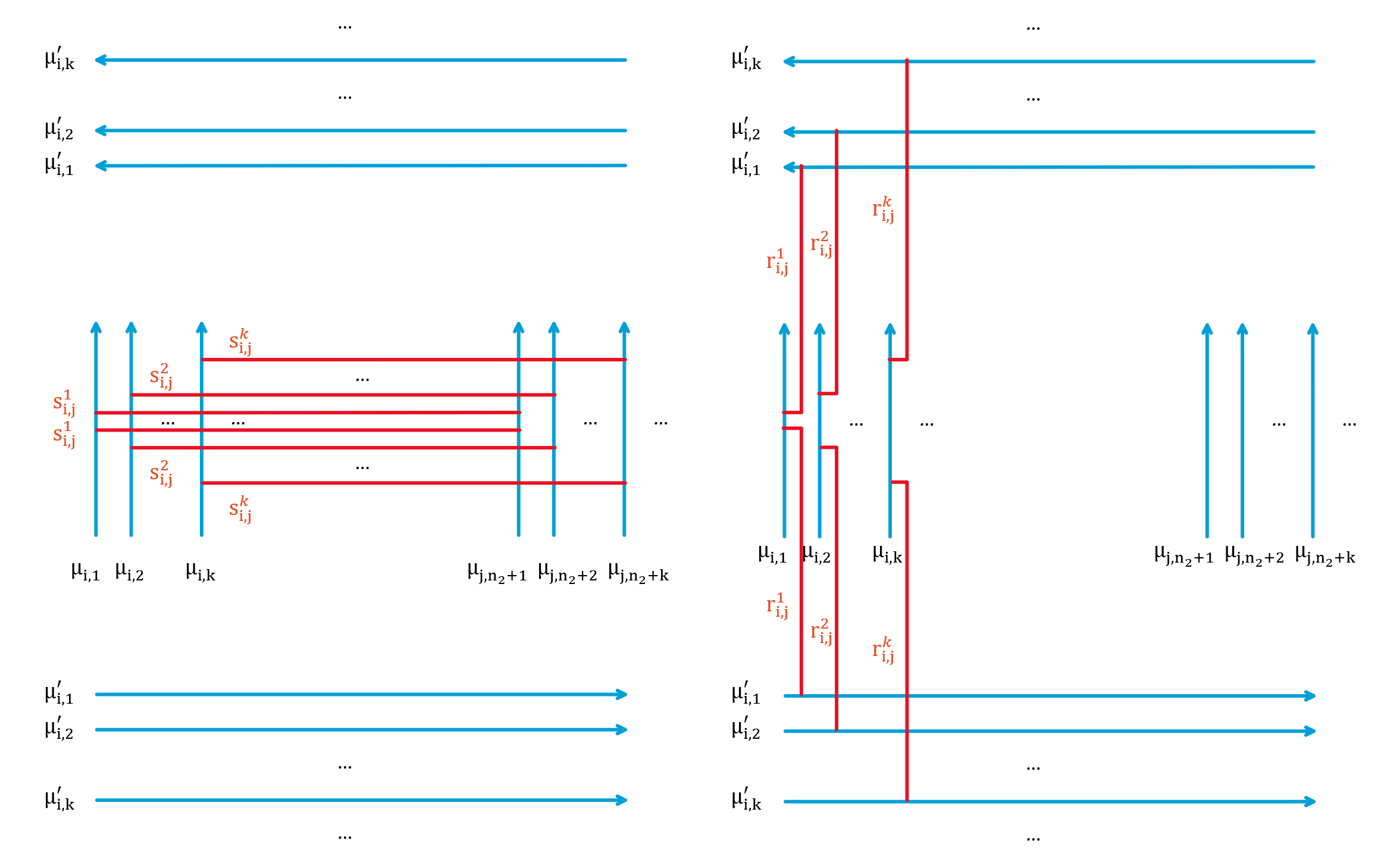}
    \caption{parallel copies of $2$-handle sliding}\label{n2hs}
\end{figure}

%Now, we are ready to define the chain map and by Proposition \ref{ordering}, we can assume we slide $\mu_1$ over $\mu_2$.

Now, we are ready to define the chain map and by Proposition \ref{ordering}, we can assume we slide $\mu_1$ over $\mu_{2}$ without the loss of generality. Consider the setup as above for $i=1$, $j=2$, we define homomorphisms $$f_n:(C_{n}(H_{g_0},U),\partial_n) \rightarrow (C_{n}(H_{g_0},U'),\partial_n')$$
such that for any $L\in X_{n}(H_{g_0},U)$, $f_{n}(L)=\sum\limits_{T\subset\vec n_{1}}(-1)^{\#T}L_{T}$, where $$\#T=|T|n_2+\sum\limits_{i=0}^{|T|-1}(n_1-t_{|T|-i}-i)=(n_1+n_2)|T|-\sum\limits_{t\in T} t-\sum\limits_{i=0}^{|T|-1}i.$$ We are going to prove that $f_n$ is a chain isomorphism.

\begin{lemma}\label{key lemma}

For $L\in X_{n}(H_{g_0},U)$ , we have $f_{n-1}\circ \partial_n=\partial_n'\circ f_n$.
    
\end{lemma}

\begin{proof}\
Consider the following:
\iffalse

    \begin{eqnarray*}
        &&f_{n-1}\circ \partial_n(L)\\
        &=&f_{n-1}(\sum\limits_{t=1}^{n_1}(-1)^t\partial_{1,t}(L)+(-1)^{n_1}\sum\limits_{t=1}^{n_2}(-1)^t\partial_{2,t}(L)+(-1)^{n_1+n_2}\sum\limits_{i=3}^{g}(-1)^{\sum\limits_{s=1}^{i-1}n_s}\sum\limits_{t=1}^{n_i}(-1)^j\partial_{i,t}(L))\\      
        && \partial_n'\circ f_{n}(L)\\  &=&\partial_n'(f_{n}(L)=\sum\limits_{T\subset\vec n_{1}}(-1)^{\#T}L_{T})\\
        &=&\sum\limits_{t=1}^{1}(-1)^t\partial_{1,t}'(L_{\emptyset})+(-1)^{1}\sum\limits_{t=1}^{n_2}(-1)^t\partial_{2,t}'(L_{\emptyset})+(-1)^{1+n_2}\sum\limits_{i=3}^{g}(-1)^{\sum\limits_{s=1}^{i-1}n_s}\sum\limits_{t=1}^{n_i}(-1)^t\partial_{i,t}'(L_{\emptyset})\\
        &+& \sum\limits_{t=1}^{0}(-1)^t(-1)^{n_2}\partial_{1,t}'(L_{\{1\}})+\sum\limits_{t=1}^{n_2+1}(-1)^t(-1)^{n_2}\partial_{2,t}'(L_{\{1\}})+(-1)^{n_2+1}\sum\limits_{i=3}^{g}(-1)^{n_2}(-1)^{\sum\limits_{s=1}^{i-1}n_s}\sum\limits_{t=1}^{n_i}(-1)^t\partial_{i,t}'(L_{\{1\}})\\       &=&-\partial_{1,1}'(L_{\emptyset})-\sum\limits_{t=1}^{n_2}(-1)^t\partial_{2,t}'(L_{\emptyset})+(-1)^{1+n_2}\sum\limits_{i=3}^{g}(-1)^{\sum\limits_{s=1}^{i-1}n_s}\sum\limits_{t=1}^{n_i}(-1)^t\partial_{i,t}'(L_{\emptyset})\\
        &+& \sum\limits_{t=1}^{n_2}(-1)^t(-1)^{n_2}\partial_{2,t}'(L_{\{1\}})+(-1)^{n_2+1}\sum\limits_{i=3}^{g}(-1)^{n_2}(-1)^{\sum\limits_{s=1}^{i-1}n_s}\sum\limits_{t=1}^{n_i}(-1)^t\partial_{i,t}'(L_{\{1\}})-\partial_{2,n+1}'(L_{\{1\}})\\
    \end{eqnarray*}

Let us first prove for the case when $n_i=0$ for $i\geq 3$:\fi
    \begin{eqnarray*}
        &&f_{n-1}\circ \partial_n(L)\\
        &=&f_{n-1}(\sum\limits_{t=1}^{n_1}(-1)^t\partial_{1,t}(L)+(-1)^{n_1}\sum\limits_{t=1}^{n_2}(-1)^t\partial_{2,t}(L)+(-1)^{n_1+n_2}\sum\limits_{i=3}^{g}(-1)^{\sum\limits_{s=1}^{i-1}n_s}\sum\limits_{t=1}^{n_i}(-1)^t\partial_{i,t}(L))\\
        &=&\sum\limits_{t=1}^{n_1}(-1)^t\sum\limits_{P\subset\overrightarrow{n_{1}-1}}(-1)^{\#P}(\partial_{1,t}(L))_P+(-1)^{n_1}\sum\limits_{t=1}^{n_2}(-1)^t\sum\limits_{T\subset\vec n_{1}}(-1)^{\#T}(\partial_{2,t}(L))_T\\&&+(-1)^{n_1+n_2}\sum\limits_{i=3}^{g}(-1)^{\sum\limits_{s=1}^{i-1}n_s}\sum\limits_{t=1}^{n_i}(-1)^t\sum\limits_{T\subset\vec n_{1}}(-1)^{\#T}(\partial_{i,t}(L))_T\\
        && \partial_n'\circ f_{n}(L)\\  &=&\sum\limits_{T\subset\vec n_{1}}(-1)^{\#T}\partial_n'(L_{T})\\
        &=& \sum\limits_{T\subset\vec n_{1}}(-1)^{\#T}\sum\limits_{t=1}^{n_1-|T|}(-1)^t\partial_{1,t}'(L_T)+\sum\limits_{T\subset\vec n_{1}}(-1)^{\#T+n_1-|T|}\sum\limits_{t=1}^{n_2+|T|}(-1)^t\partial_{2,t}'(L_T)\\&&+\sum\limits_{T\subset\vec n_{1}}(-1)^{\#T}(-1)^{n_1+n_2}\sum\limits_{i=3}^{g}(-1)^{\sum\limits_{s=1}^{i-1}n_s}\sum\limits_{t=1}^{n_i}(-1)^t\partial_{i,t}'(L_T)  
    \end{eqnarray*}
Notice that most of the terms cancel pairwise as follows: 
\begin{enumerate}
    \item For differential taking on $\mu_i$ with $i\geq 3$, we have obvious cancellation by definitions:\begin{eqnarray*}
&&(-1)^{n_1+n_2}\sum\limits_{i=3}^{g}(-1)^{\sum\limits_{s=1}^{i-1}n_s}\sum\limits_{t=1}^{n_i}(-1)^t\sum\limits_{T\subset\vec n_{1}}(-1)^{\#T}(\partial_{i,t}(L))_T\\
&=&
\sum\limits_{T\subset\vec n_{1}}(-1)^{\#T}(-1)^{n_1+n_2}\sum\limits_{i=3}^{g}(-1)^{\sum\limits_{s=1}^{i-1}n_s}\sum\limits_{t=1}^{n_i}(-1)^t\partial_{i,t}'(L_T)    \end{eqnarray*}
\item For differential taking on $\mu_2$ for the first $n_2$ places:
\begin{eqnarray*}
&&(-1)^{n_1}\sum\limits_{t=1}^{n_2}(-1)^t\sum\limits_{T\subset\vec n_{1}}(-1)^{\#T}(\partial_{2,t}(L))_T\\
&=&
\sum\limits_{T\subset\vec n_{1}}(-1)^{\#T+n_1-|T|}\sum\limits_{t=1}^{n_2}(-1)^t\partial_{2,t}'(L_T)    \end{eqnarray*}
We have $(\partial_{2,t}(L))_T=\partial_{2,t}'(L_T)$ since, by our definition, $f_n$ only sends the part of $L$ attached to $\mu_1$ to $\mu_2$ with order higher than those $n_2$ pairs of framed points originally occupied by $L$. Thus, we only need to compare the signs for the same $t$ and $T$ and they match exactly as follows:
\begin{enumerate}
    \item  for the first row, the exponent is: 
$n_1+t+|T|(n_2-1)+\sum\limits_{i=0}^{|T|-1}(n_1-t_{|T|-i}-i),$ 
    \item for the second row, the exponent is: 
$|T|n_2+\sum\limits_{i=0}^{|T|-1}(n_1-t_{|T|-i}-i)+n_1-|T|+t.$ 
\end{enumerate}
\end{enumerate}

Now, we only need to show $$\sum\limits_{t=1}^{n_1}(-1)^t\sum\limits_{P\subset\overrightarrow{n_{1}-1}}(-1)^{\#P}(\partial_{1,t}(L))_P=\sum\limits_{T\subset\vec n_{1}}(-1)^{\#T}\sum\limits_{t=1}^{n_1-|T|}(-1)^t\partial_{1,t}'(L_T)+\sum\limits_{T\subset\vec n_{1}}(-1)^{\#T+n_1-|T|}\sum\limits_{t=n_2+1}^{n_2+|T|}(-1)^t\partial_{2,t}'(L_T).$$
Observe that, for a fixed $P$, we have:
\begin{equation}\label{keyeq}
(\partial_{1,t}(L))_P=\partial_{1,t-<t,P>}'(L_{P[t]})+\partial_{2,n_2+1+<t,P>}'(L_{P[t]\cup \{t\}}),    
\end{equation}
 where $<t,P>$ is the integer counting how many elements in $P$ that are less than $t$; $P[t]$ is the set obtained from $P$ by replacing each element $x\in P$ greater than or equal to $t$ by $x+1$. To aid the understanding of equation (\ref{keyeq}), we show a demonstration of $f_2$. There are two terms on the left hand side and four terms on the right hand side.  Actually two terms on the right hand side cancels with each other, and the remaining two terms on the right hand side are the same as the two terms on the left hand side one by one(we will demonstrate the cancellation using Figure \ref{hs}).  

 Let us represent $L$ by Figure \ref{L}. This is a local picture inside a collar neighborhood of the small disk where we move $\alpha_1$ and $\alpha_2$ together. The black part is the relative link $L$ attached to $a_{1,1}$, $b_{1,1}$ and  $a_{1,2}$, $b_{1,2}$, and the relative link $L$ goes outside the collar neighborhood through the black square. Here inside the paper is the handlebody and all endpoints of $L$ are considered to be attached to boundary of the handlebody upwards.  Figure \ref{f2} is $f_2(L)$ and Figure \ref{partial2} is $\partial_2(L)$.
\begin{figure}[H]
    \centering
    \includegraphics[width=0.3\linewidth]{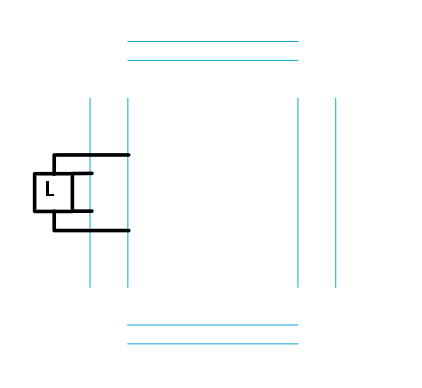}
    \caption{$L$}\label{L}
\end{figure}

\begin{figure}[H]
    \centering
    \includegraphics[width=1\linewidth]{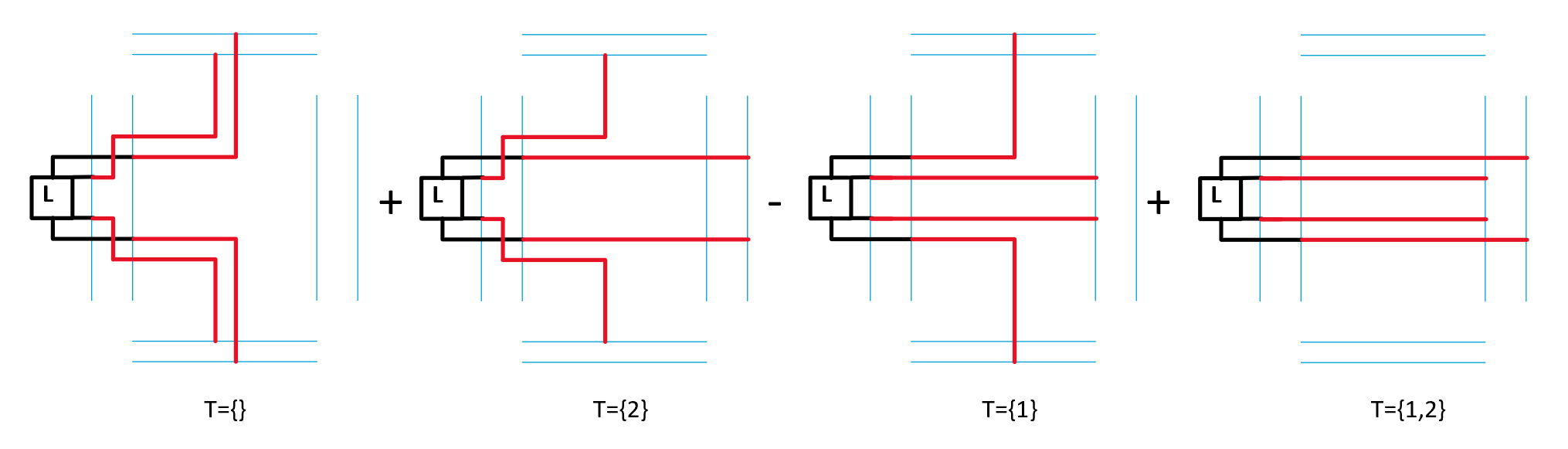}
    \caption{$f_2(L)$}\label{f2}
\end{figure}

\begin{figure}[H]
    \centering
    \includegraphics[width=1\linewidth]{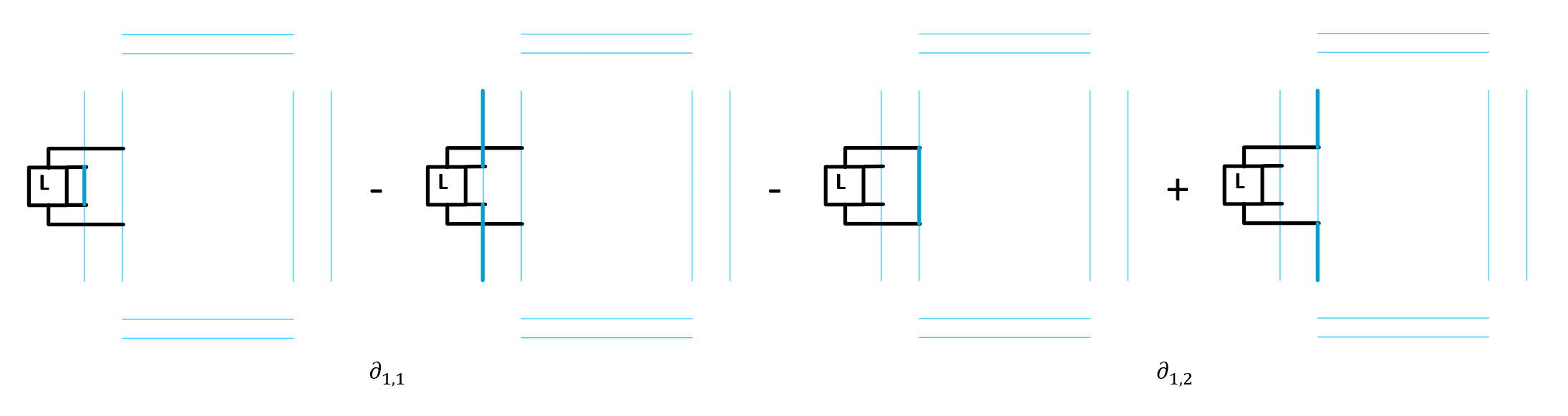}
    \caption{$\partial_2(L)$}\label{partial2}
\end{figure}
In Figure \ref{hs}, red arcs are the effect of $f_2$ and blue arcs are the effect of taking differential. On the left we have $\partial_2'\circ f_2(L)$, where red and blue are above black and blue is above red, and on the right we have $f_1\circ\partial_2(L)$, where red and blue are above black and red is above blue. We have:
\begin{enumerate}
    \item $(\partial_{1,1}(L))_{\emptyset}=\partial_{1,1}'(L_{\emptyset})+\partial_{2,1}'(L_{\{1\}})$
    \item $(\partial_{1,1}(L))_{\{1\}}=\partial_{1,1}'(L_{\{2\}})+\partial_{2,1}'(L_{\{1,2\}})$
    \item $(\partial_{1,2}(L))_{\emptyset}=\partial_{1,2}'(L_{\emptyset})+\partial_{2,1}'(L_{\{2\}})$
    \item $(\partial_{1,2}(L))_{\{1\}}=\partial_{1,1}'(L_{\{1\}})+\partial_{2,2}'(L_{\{1,2\}})$
\end{enumerate}

For $(4)$, we see that the first term of $\partial_{1,1}'(L_{\{1\}})$ cancels with the second term of $\partial_{2,2}'(L_{\{1,2\}})$, where the isotopy occurs in the Northeast and Southeast corners; the second term of $\partial_{1,1}'(L_{\{1\}})$ cancels with the second term of $(\partial_{1,2}(L))_{\{1\}}$, where the isotopy occurs in the Northwest and Southwest corners; and the first term of $\partial_{2,2}'(L_{\{1,2\}})$ cancels with the first term of $(\partial_{1,2}(L))_{\{1\}}$, where we can pull the blue arc in the first term of $\partial_{2,2}'(L_{\{1,2\}})$ back to the position of the blue arc in the first term of $(\partial_{1,2}(L))_{\{1\}}$ below the framed points $a_{2,1}$ and $b_{2,1}$\footnote{This is why we choose the setup so at the beginning of this section.}.
\begin{figure}[H]
    \centering
    \includegraphics[width=1\linewidth]{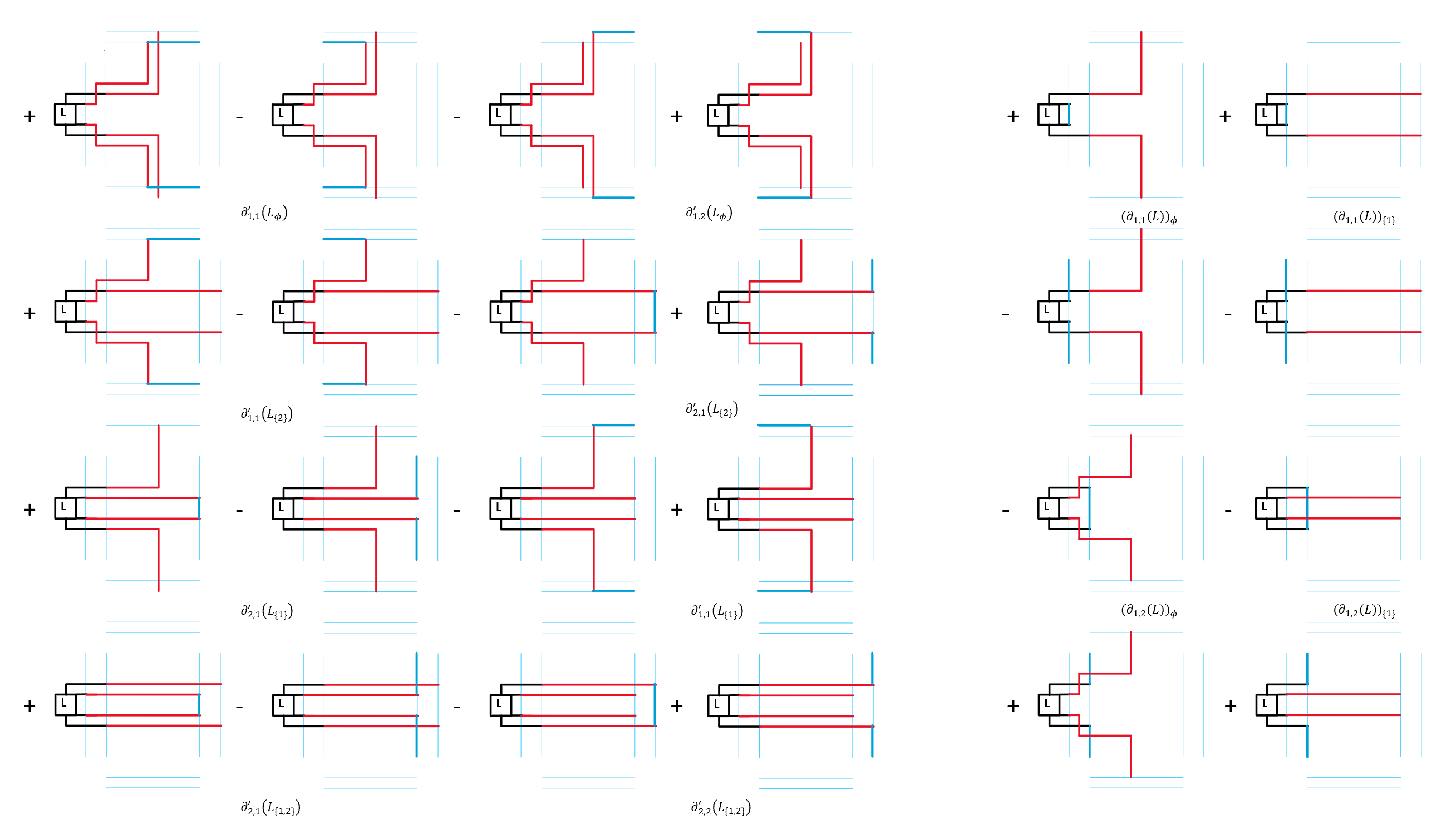}
    \caption{$\partial_2'\circ f_2=f_1\circ\partial_2$}\label{hs}
\end{figure}
 For the sign, let us compute the exponents of the signs, which shows they have the same sign:
 \begin{enumerate}
     \item for $(\partial_{1,t}(L))_P$: 
     \begin{eqnarray*}
     &&t+(n_1-1+n_2)|P|-\sum\limits_{a\in P}a-\sum\limits_{i=0}^{|P|-1}i\\
     &=&t+(n_1+n_2)|P|-|P|-\sum\limits_{a\in P}a-\sum\limits_{i=0}^{|P|-1}i
     \end{eqnarray*}
     \item for $\partial_{1,t-<t,P>}'(L_{P[t]})$:
     \begin{eqnarray*}
     &&t-<t,P>+(n_1-|P[t]|+n_2+|P[t]|)|P[t]|-\sum\limits_{b\in P[t]}b-\sum\limits_{i=0}^{|P|-1}i\\
     &=&t-<t,P>+(n_1+n_2)|P|-(\sum\limits_{a\in P}a+|P|-<t.P>)-\sum\limits_{i=0}^{|P|-1}i\\
     &=&t+(n_1+n_2)|P|-\sum\limits_{a\in P}a-|P|-\sum\limits_{i=0}^{|P|-1}i
     \end{eqnarray*}  
     \item for $\partial_{2,n_2+1+<t,P>}'(L_{P[t]\cup \{t\}})$:
     \begin{eqnarray*}
     &&n_2+1+<t,P>+n_1-(|P|+1)+(n_1-|P|-1+n_2+|P|+1)(|P|+1)-\sum\limits_{c\in P[t]\cup\{t\}}c-\sum\limits_{i=0}^{|P|}i\\
     &=&(n_1+n_2)(|P|+2)-|P|-(\sum\limits_{a\in P}a+|P|-<t.P>+t)-\sum\limits_{i=0}^{|P|-1}i-|P|\\
     &=&-t+(n_1+n_2)(|P|+2)-\sum\limits_{a\in P}a-3|P|-\sum\limits_{i=0}^{|P|-1}i
     \end{eqnarray*}
 \end{enumerate}
Therefore, we finish the proof.
\end{proof}

\begin{theorem}

Let $(H_{g_0},U)$ and $(H_{g_0},U')$ be as above. The two chain complexes are chain isomorphic.
    
\end{theorem}

\begin{proof}
By Lemma \ref{key lemma}, it remains to show  $$f_n:(C_{n}(H_{g_0},U),\partial_n) \rightarrow (C_{n}(H_{g_0},U'),\partial_n')$$
is a bijection.
\begin{enumerate}
    \item Injectivity: suppose $f_n(L_1)=f_n(L_2)$, which, in particular, implies $(L_1)_{\emptyset}=(L_2)_{\emptyset}$ so that we have $L_1=L_2$.
    \item Surjectivity: notice that $f_n$ sends a relative framed ink $L\in C_{n}(H_{g_0},U)$ to a linear combination of elements in $C_{n}(H_{g_0},U')$ with number of pairs occupied on $B_1'$ less than or equal to $n_1(L)$, so we have filtrations on $F_{p}C_{n}(H_{g_0},U)$ and 
$F_{p}C_{n}(H_{g_0},U')$ respect our chain map $f_n$, where $F_{p}C_{n}(H_{g_0},U)$ and 
$F_{p}C_{n}(H_{g_0},U')$ consisting of span of relative framed links whose $n_1\leq p$. We do induction on $p$. Clearly, when $p=0$, $f_n$ restricting on $F_0C_{n}(H_{g_0},U)$ is surjective. Suppose $f_{n}$ restricting on $F_pC_{n}(H_{g_0},U)$ is surjective for $p<k$. When $p=k$, for any relative framed link $L'\in F_kC_{n}(H_{g_0},U')$, if $n_1(L')<k$, by induction hypothesis, we are done; if  $n_1(L')=k$, then $L'$ can be viewed as $L_{\emptyset}$ for some $L\in F_kC_{n}(H_{g_0},U)$, then $f_n(L)=L'+y$, where $y\in F_{k-1}C_{n}(H_{g_0},U')$ is a linear combination of relative framed links with $n_1<k$. By induction hypothesis,  there exist $x\in F_{k-1}C_{n}(H_{g_0},U')$ such that $f_{n}(x)=y$. Thus, $f_n(L-x)=L'$ restricting on $F_kC_{n}(H_{g_0},U)$ is surjective. Since $f_n$ restricting on $F_pC_{n}(H_{g_0},U)$ for $p=n$ is $f_n$, $f_n$ is surjective for all $n$.
\end{enumerate}

\end{proof}

\section{Stabilization}
In this section, we discuss how stabilization affects on the homology.  Let $(H_{g_0},U)$ be any Heegaard splitting, and $(H_{g_0+1},U')$ be the Heegaard splitting obtained from $(H_{g_0},U)$ by applying one stabilization, where the new attaching curve is denoted by $\mu_{g_0+1}$ with an arbitrary chosen orientation and based points. Then the embedding $\iota:H_{g_0}\hookrightarrow H_{g_0+1}$(see Figure \ref{emb}) induces an inclusion $\iota_{\#}:C_{*}(H_{g_0},U)\rightarrow C_{*}(H_{g_0+1},U')$.
\begin{theorem}
The natural inclusion $\iota_{\#}:C_{*}(H_{g_0},U)\rightarrow C_{*}(H_{g_0+1},U')$ induces a monomorphism on the level of homology.   
\end{theorem}

\begin{figure}[htbp]
    \centering \includegraphics[width=0.7\linewidth]{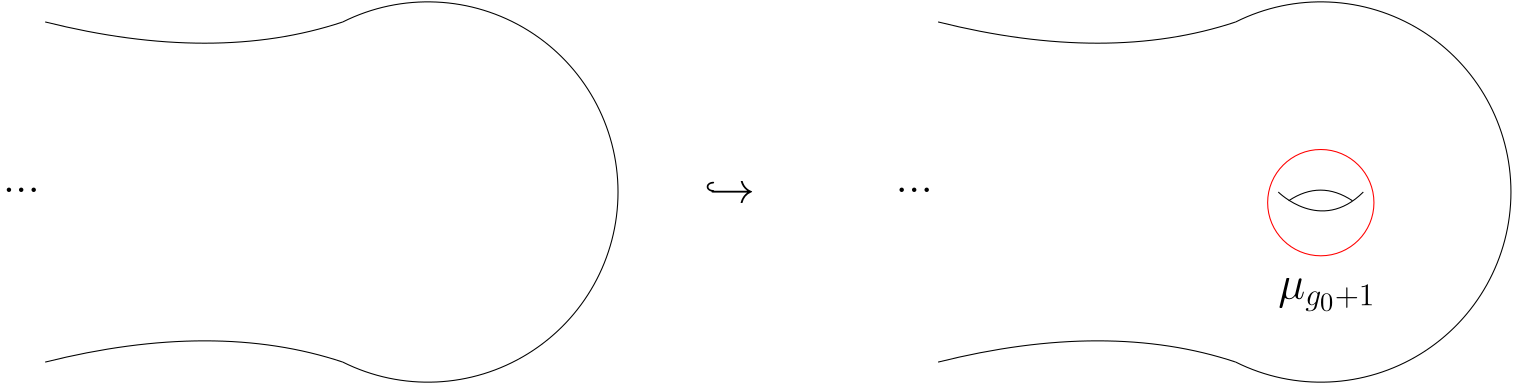} \caption{the embedding of $H_{g_0}$ into $H_{g_0+1}$}\label{emb}
\end{figure}
\begin{proof}
\vspace{-10pt}
By definition, $C_{n}(H_{g_0},U)$ embeds into $C_{n}(H_{g_0+1},U')$ as a chain subcomplex, so we have the standard short exact sequence of chain complexes:
$$0\rightarrow C_{n}(H_{g_0},U) \xrightarrow{f} C_{n}(H_{g_0+1},U')\rightarrow C_{n}(H_{g_0+1},U')/ C_{n}(H_{g_0},U)\rightarrow 0 $$
where $f$ is the homomorphism induced by the embedding of $H_{g_0}$ into $H_{g_0+1}$. To prove the statement, we only need to find a chain map $g:C_{n}(H_{g_0+1},U')\rightarrow C_{n}(H_{g_0},U)$ such that $g\circ f=id$. We define
\[
\begin{array}{r c c c}
  g\colon & C_{n}(H_{g_0+1},U') & \to      & C_{n}(H_{g_0},U) \\
           & L                   & \mapsto  & g(L)
\end{array}
\]
such that
\begin{enumerate}
    \item if at least one pair of endpoints of $L$ is attached to $\mu_{g_0+1}$, $g(L)=0$;
    \item otherwise, when we fill the $g_0+1$'th genus, $H_{g_0+1}$ will become homeomorphic to $H_{g_0}$. $g(L)$ will be the element obtained from $L$ during the above procedure (compare Figure \ref{g}).
\end{enumerate}

\begin{figure}[H]
    \centering
    \includegraphics[width=0.8\linewidth]{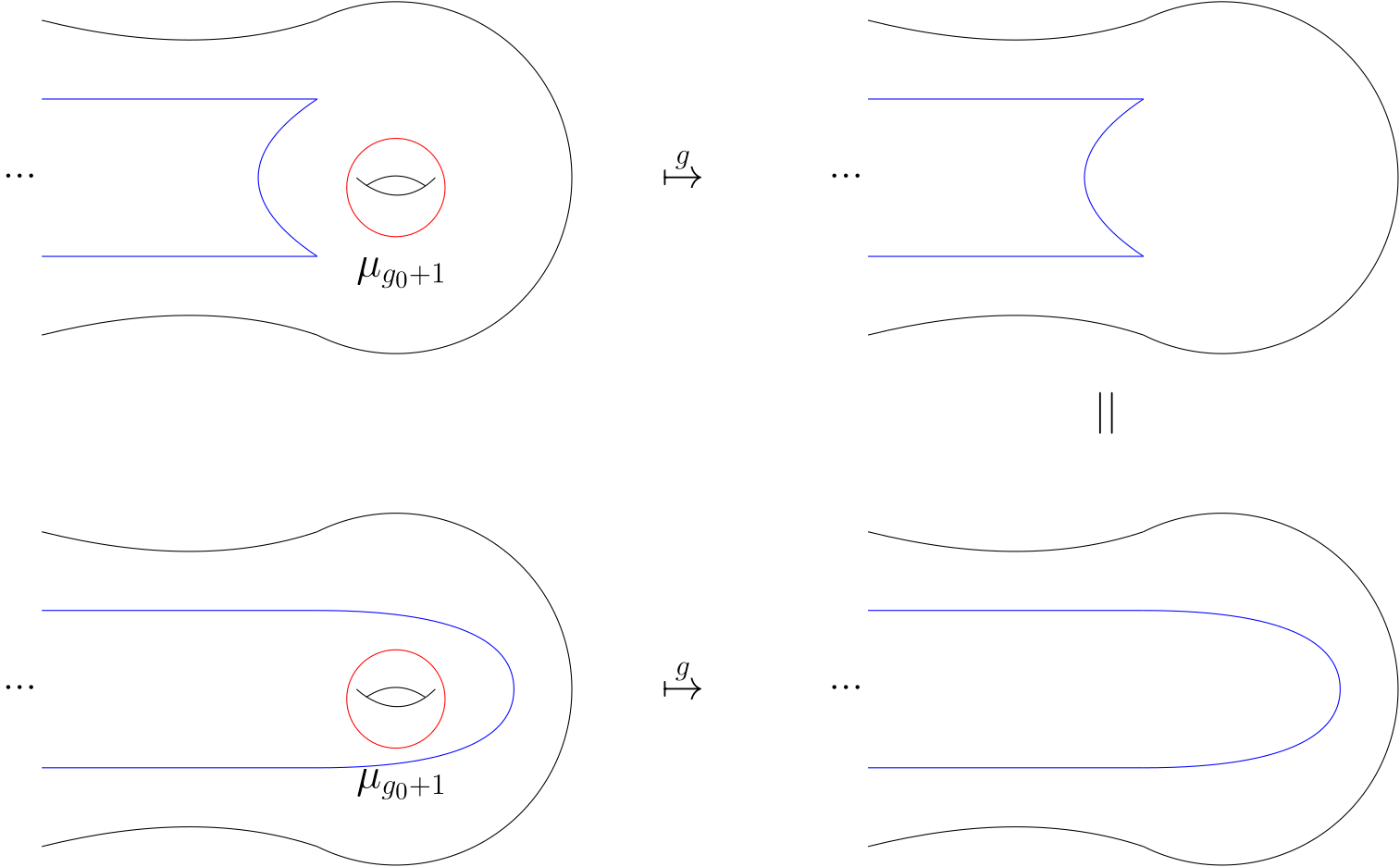}
    \caption{$g$}\label{g}
\end{figure}

Clearly, $g\circ f=id$ since graphically, we just add a genus and then remove it.
Next, let us check that $g$ is a chain map, i.e. we need to compare $\partial_n\circ g$ and $g\circ \partial_n$.
\begin{enumerate}
    \item[1.] Suppose there are more than one pair of endpoints of $L$ connecting to $\mu_{g_0+1}$, we have $\partial_n\circ g(L)=0$. As for $g\circ \partial_n(L)$, $\partial_n(L)$ is a linear combination of elements that has at least one pair of endpoints connecting to $\mu_{g_0+1}$, so $g\circ \partial_n(L)=0$; 
    \item[2.] Suppose that $L$ has only one pair of endpoints connecting to $\mu_{g_0+1}$,  we have $\partial_n\circ g(L)=0$. We have $\partial_n(L)$ is a linear combination of elements that has at least one pair of endpoints connecting to $\mu_{g_0+1}$ except for two terms coming from the differentials concerning $\mu_{g_0+1}$. However, those two terms cancels with each other since we have filled the last genus(see Figure \ref{gpartial}), so $g\circ \partial_n(L)=0$; 
    \item[3.]  Suppose that $L$ has no pair of endpoints connecting to $\mu_{g_0+1}$, the differential has nothing to do with $\mu_{g_0+1}$. After taking the differential and filling the last genus, clearly $\partial_n\circ g=g\circ \partial_n$.
\end{enumerate}
\begin{figure}[H]
    \centering \includegraphics[width=0.9\linewidth]{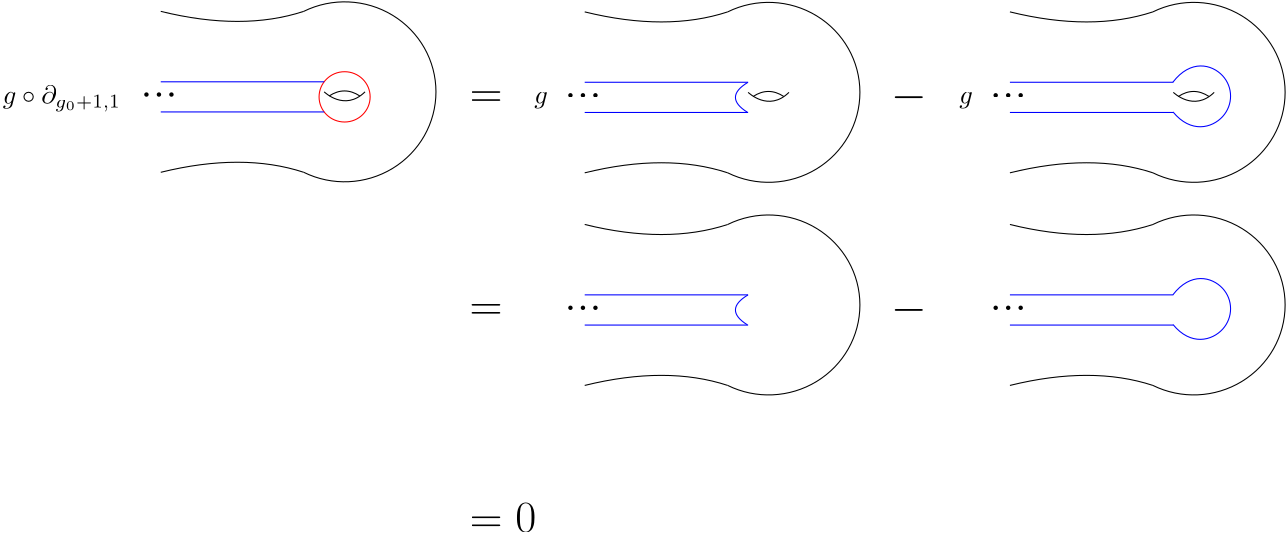} \caption{$g\circ\partial_{g_0+1}(L)$}\label{gpartial}
\end{figure}
\end{proof}

The natural inclusion $\iota_{\#}:C_{*}(H_{g_0},U)\rightarrow C_{*}(H_{g_0+1},U')$ in general does not induce an isomorphism on the level of homology. 
For instance, consider the stabilization of the genus zero Heegaard splitting of $S^3$. The rank of the first homology after stabilization is $1$, while it is zero before stabilization. See Proposition \ref{p1}. 

\section{Examples}
In this section, we give some explicit computation about the genus one Heegaard splittings of the lens spaces $L(p,1)$. 
Let $S_q$ denote the Chebyshev polynomials of the second kind, which satisfy the recurrence relation $S_{q+1}(x) = xS_q(x) - S_{q-1}(x)$, with the initial conditions $S_{0}(x) = 1$ and $S_{1}(x) = x$. It is convenient for us to extend the initial conditions to include negative indices, in which case $S_{-2}(x) = -1$ and $S_{-1}(x) = 0$. Notice that we indeed have $S_{n-2}(x)=-S_{-n}(x)$.
\begin{proposition}\label{p1}

Consider the gunes one Heegaard splitting $(H_1,U)$ of lens spaces $L(p,1)$, where we choose the framed base points $a_{1,i}$, $b_{1,i}$ as shown in Figure \ref{ep1}(this is the case for $L(2,1)$, which is enough to demonstrate the computation) , we have:
\begin{eqnarray*}
    H_{0}((H_1,U))\cong \bigoplus\limits_{i=1}^{\lfloor{\frac{p}{2}}\rfloor+1}\mathbb{Z}[A,A^{-1}];\\
     H_{1}((H_1,U))\cong \bigoplus\limits_{i=1}^{\lfloor{\frac{p}{2}}\rfloor+1}\mathbb{Z}[A,A^{-1}].
\end{eqnarray*}
\end{proposition}
Consider Figure \ref{ep1}, let $c_{m}$ denote the relative curve connecting $a_{1,1}$ and $b_{1,1}$ rotating clockwise $m$ time around the black circle in the center for nonnegative $m$ and rotating counterclockwise $|m|$ time around the black circle in the center for negative $m$. According to Theorem \ref{rkbsmsib}, $\{c_m|m\in\mathbb{Z}\}$ forms a basis for $C_1(H_1,U)$, so we compute $\partial_1(c_m)$ as follows:
\begin{figure}[H]
    \centering
    \includegraphics[width=0.6\linewidth]{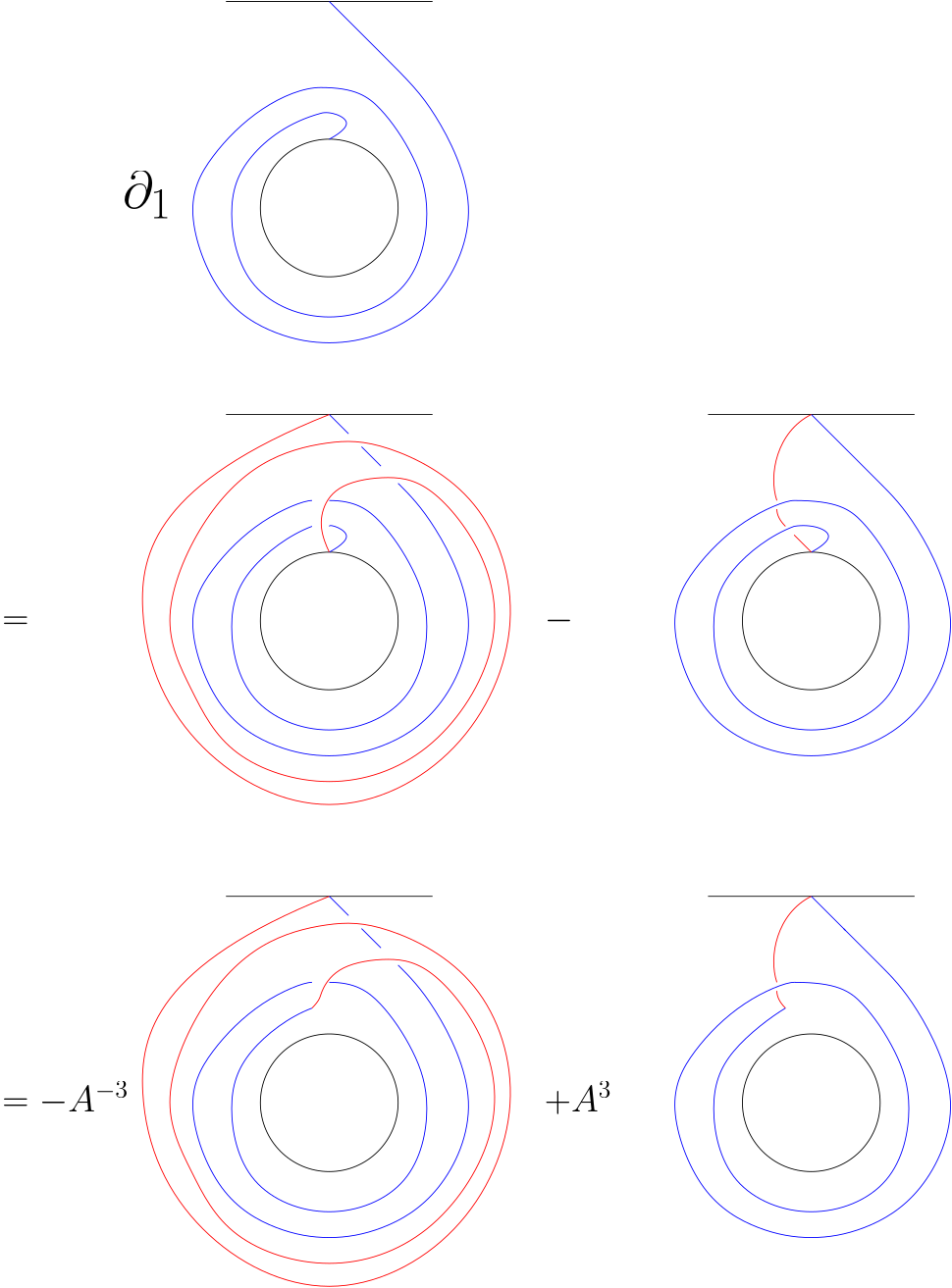}
    \caption{$\partial_1(c_{2})$}\label{ep1}
\end{figure}
According to Lemma 3.6-3.9 in \cite{BKSW}, 
\begin{eqnarray*}
    \partial_{1}(c_m)&=&-A^{-3}N(m+p,0)+A^{3}P(m,0)\\
    &=&-A^{-3}(A^{-m-p+1}S_{m+p}(z)-A^{-m-p+5}S_{m+p-2}(z))+A^{3}(A^{m-1}S_{m}(z)-A^{m-5}S_{m-2}(z))\\
    &=&-A^{-m-p-2}S_{m+p}(z)+A^{-m-p+2}S_{m+p-2}(z)+A^{m+2}S_{m}(z)-A^{m-2}S_{m-2}(z),
\end{eqnarray*}
where $z$ is the boundary parallel circle in the annulus.
\begin{lemma}
When $p$ is odd, $ker\partial_1$ is generated by $X_p=\{c_m+c_{-m-p}|m\in \mathbb{Z}\}$; and when $p$ is even, $ker\partial_1$ is generated by $X_p=\{c_m+c_{-m-p}|m\in \mathbb{Z},m\neq-\frac{p}{2}\}\cup\{c_{-\frac{p}{2}}\}$
    
\end{lemma}
\begin{proof}
\begin{enumerate}
    \item Step 1. Any element in $X_p$ is a cycle:$$\partial_1(c_{-m-p})= -A^{m-2} S_{-m}(z) + A^{m+2} S_{-m-2}(z) + A^{-m-p+2} S_{-m-p}(z) - A^{-m-p-2} S_{-m-p-2}(z)$$
Using $S_{-n}(z) = -S_{n-2}(z)$, we obtain
$$\partial_1(c_{-m-p})= A^{m-2} S_{m-2}(z) - A^{m+2} S_m(z) - A^{-m-p+2} S_{m+p-2}(z) + A^{-m-p-2} S_{m+p}(z).$$
Thus $\partial_1(c_m+c_{-m+p})=0$. When $p$ is even, by a direct computation, $\partial_{1}(c_{-\frac{p}{2}})=0$.
\item Step 2. Every cycle has this form:\\
Let \(X = \sum_{m} x_m c_m\) be an arbitrary cycle. 
Collecting the coefficients of \(S_k(z)\) in \(\partial_1(X)\) gives
\[
\partial_1(X) = \sum_{k \in \mathbb{Z}} C_k S_k(z),
\]
where
$$
C_k = -x_{k-p} A^{-k-2} + x_{k-p+2} A^{-k} + x_k A^{k+2} - x_{k+2} A^{k}.
$$
Because the \(S_k(z)\) satisfy \(S_{-n}(z) = -S_{n-2}(z)\) and \(S_{-1}(z)=0\), the module they generate is freely spanned by \(\{S_k(z)\}_{k \ge 0}\). 
Combining the coefficients, we have the coefficient of \(S_k(z)\) (\(k \ge 0\)) is indeed \(C_k - C_{-k-2}\). 
Hence \(\partial_1(X) = 0\) is equivalent to
\[
C_k = C_{-k-2} \quad \text{for all } k \ge 0.
\]
Write both sides explicitly:
\begin{align*}
C_k &= -x_{k-p} A^{-k-2} + x_{k-p+2} A^{-k} + x_k A^{k+2} - x_{k+2} A^{k}, \\
C_{-k-2} &= -x_{-k-p-2} A^{k} + x_{-k-p} A^{k+2} + x_{-k-2} A^{-k} - x_{-k} A^{-k-2}.
\end{align*}
Since \(A\) is an indeterminate, the coefficients of each distinct power of \(A\) must vanish. 
For \(k \ge 1\) the powers \(A^{-k-2}, A^{-k}, A^{k}, A^{k+2}\) are all distinct, giving four families of equations:
\begin{align*}
\text{(1)}\quad & x_{-k-p} = x_k && \text{(from } A^{k+2}) \\
\text{(2)}\quad & x_{-k} = x_{k-p} && \text{(from } A^{-k-2}) \\
\text{(3)}\quad & x_{k-p+2} = x_{-k-2} && \text{(from } A^{-k}) \\
\text{(4)}\quad & x_{-k-p-2} = x_{k+2} && \text{(from } A^{k}) 
\end{align*}
For \(k = 0\) the distinct powers are \(A^{-2}, A^0, A^2\) and we obtain
\[
x_0 = x_{-p} \quad\text{and}\quad x_{-p+2} - x_2 + x_{-p-2} - x_{-2} = 0.
\]

We now show that these relations force \(x_m = x_{-m-p}\) for every integer \(m\).
\begin{itemize}
\item If \(m \ge 0\), take \(k = m\) in $(1)$ (and \(k=0\) for \(m=0\)) to get \(x_{-m-p} = x_m\).
\item If \(m \le -p\), let \(k = -m-p \ge 0\); then $(1)$(and \(k=0\) for \(-m-p=0\)) gives \(x_{-(-m-p)-p} = x_{-m-p}\), i.e.\ \(x_m = x_{-m-p}\).
\item If \(-p < m < 0\) , put \(k = -m\). Then \(1 \le k \le p-1\) and $(2)$ gives \(x_{-k} = x_{k-p}\), i.e.\ \(x_m = x_{-m-p}\).
\end{itemize}
Thus every cycle satisfies \(x_m = x_{-m-p}\) for all \(m \in \mathbb{Z}\). 
Therefore the kernel of \(\partial_1\) is exactly the subspace spanned by $X_p$.
\end{enumerate}
\end{proof}
\newpage
For the second boundary map, we first consider the basis of $C_2(H_1,U)$ as shown in Figure \ref{c2base}:
\begin{figure}[H]
    \centering
    \includegraphics[width=0.6\linewidth]{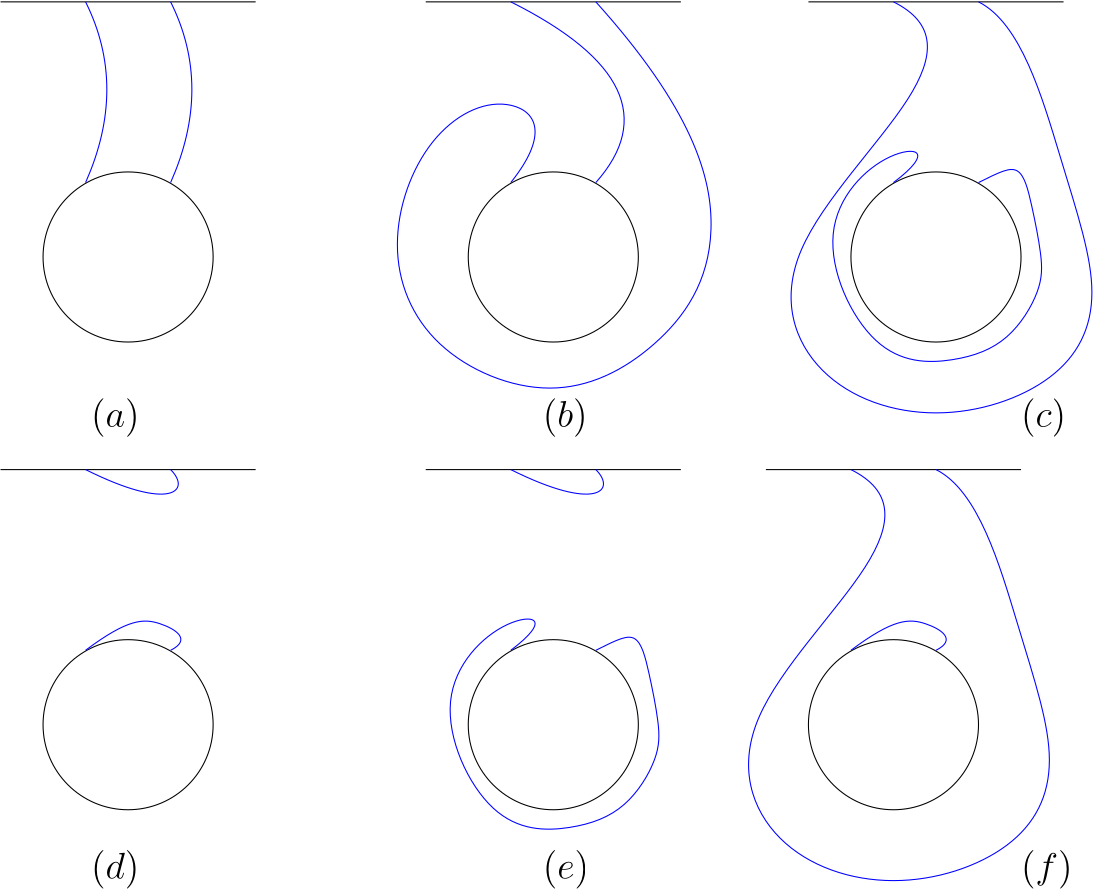}
    \caption{Basis for $C_2(H_1,U)$}\label{c2base}
\end{figure}
All the basis of $C_2(H_1,U)$ consists of $t^{k}(a),t^{k}(b)$ with $t$ the Dehn twist around the black circle in the center, and $z^m(c),z^m(d),z^m(e),z^m(f)$. The following Figures demonstrate that $\partial_2(t^{k}(a))=\partial_2(t^{k}(b))=\partial_2(z^m(c))=\partial_2(z^m(d))=0.$
\begin{figure}[H]
    \centering
    \includegraphics[width=0.6\linewidth]{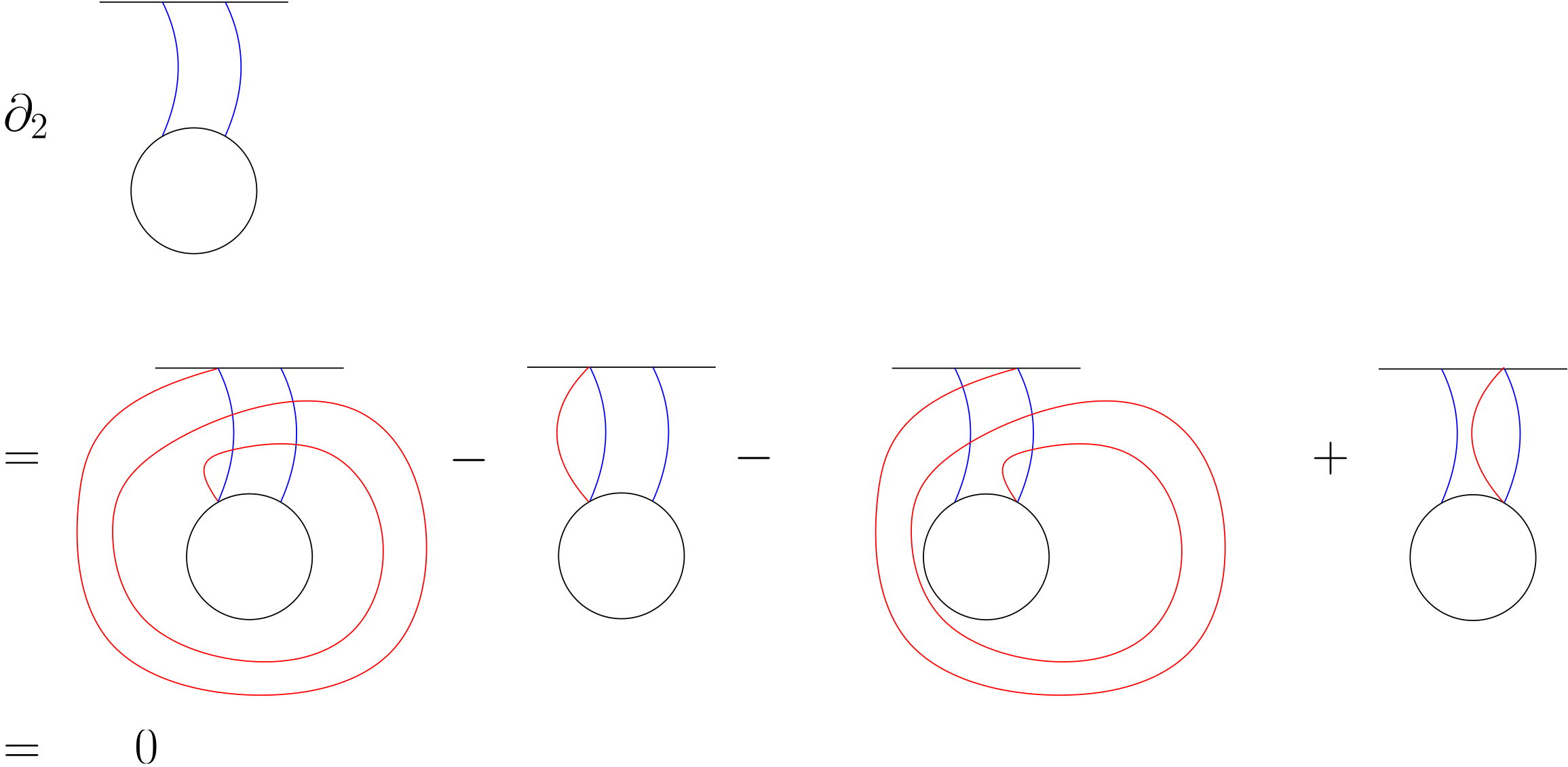}
    \caption{$\partial_2((a))$}\label{0000}
\end{figure}
\begin{figure}[H]
    \centering
    \includegraphics[width=0.6\linewidth]{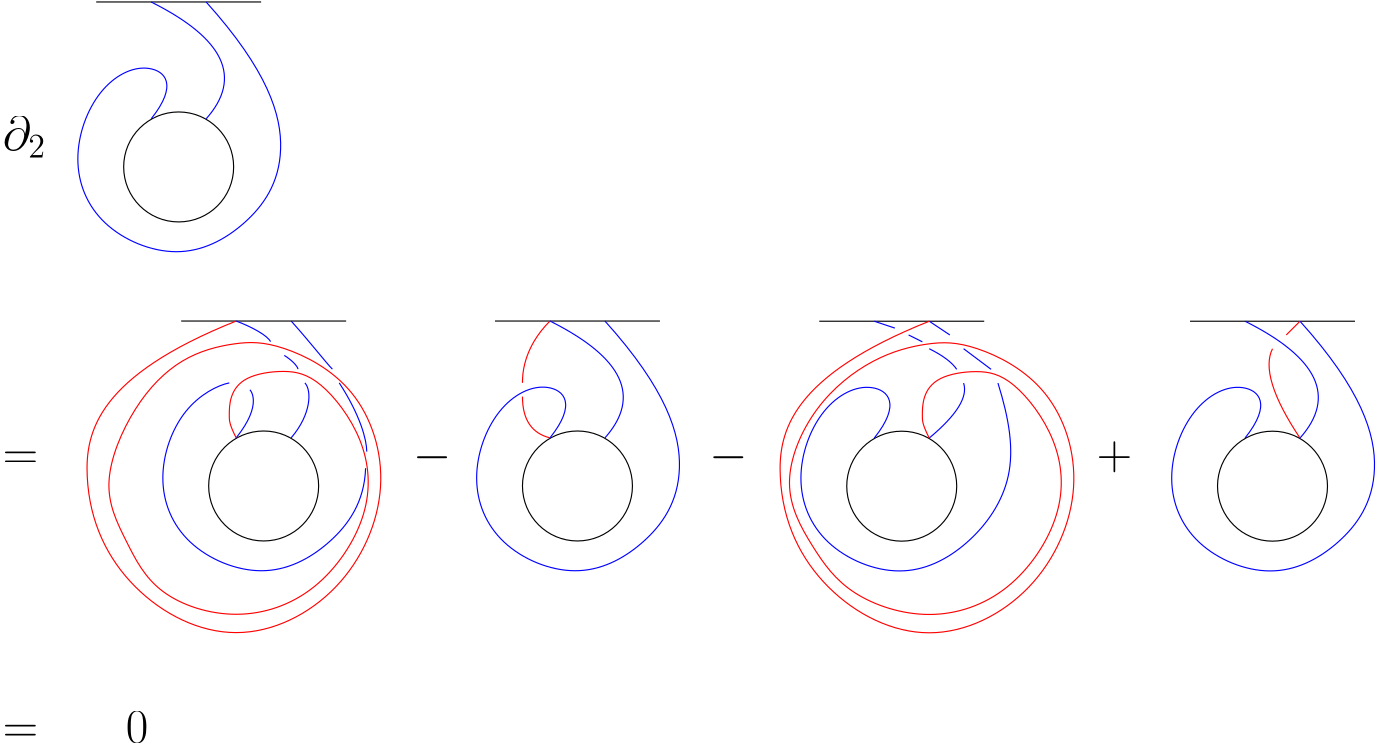}
    \caption{$\partial_2((b))$}\label{0}
\end{figure}
\begin{figure}[H]
    \centering
    \includegraphics[width=0.62\linewidth]{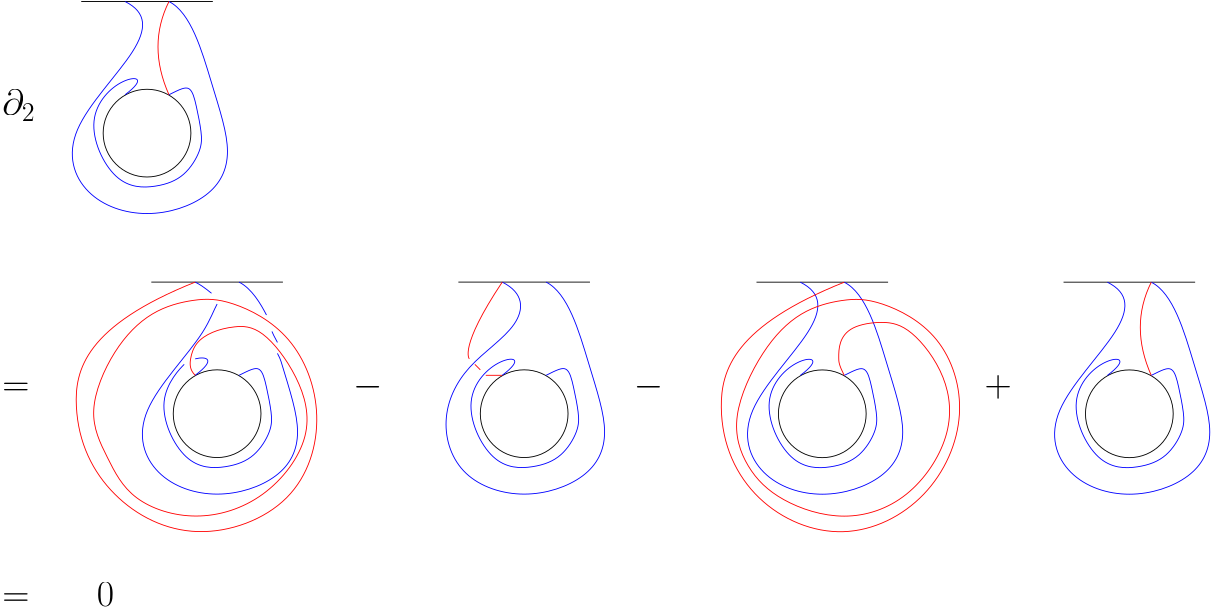}
    \caption{$\partial_2((c))$}\label{00}
\end{figure}
\begin{figure}[H]
    \centering
    \includegraphics[width=0.6\linewidth]{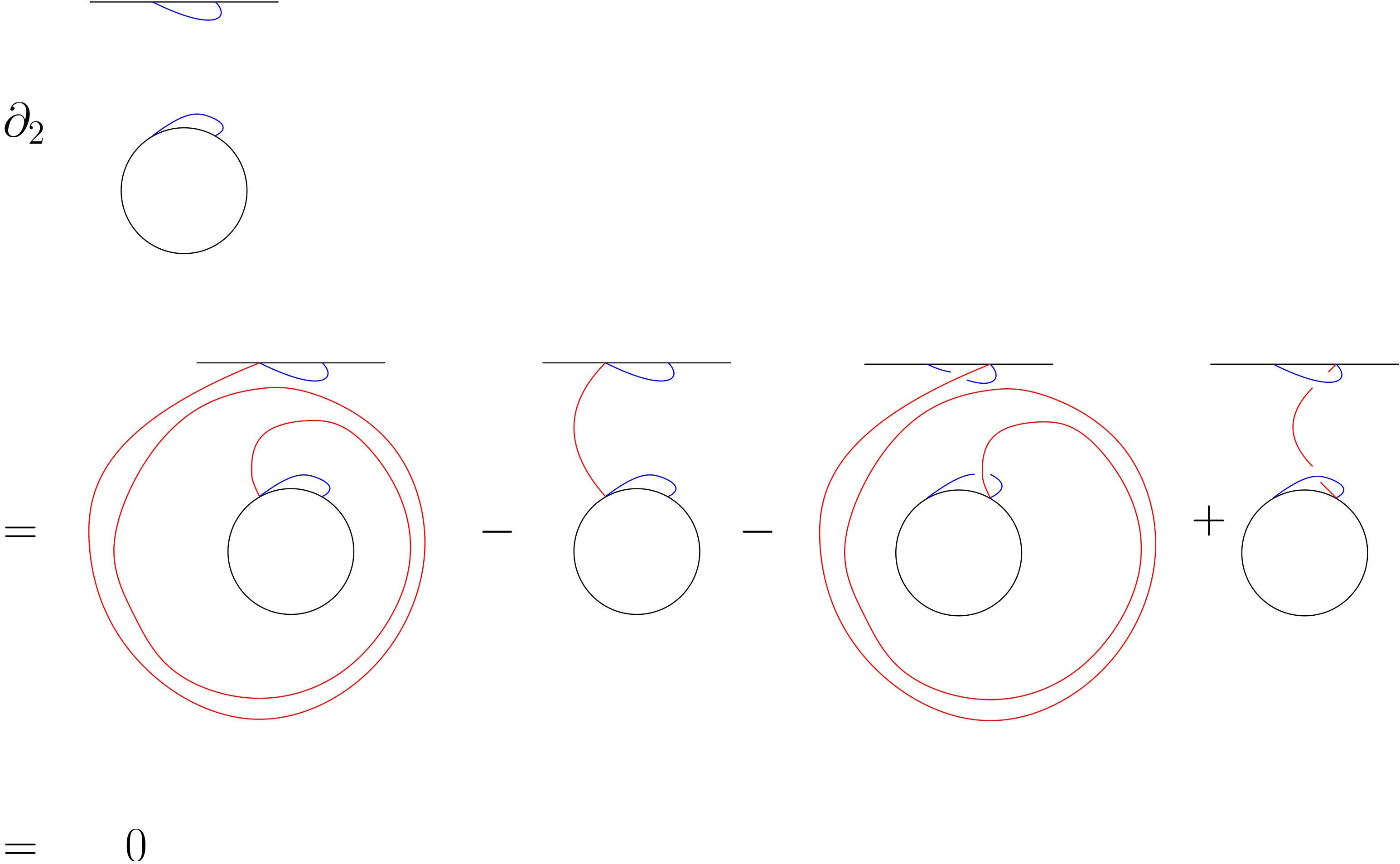}
    \caption{$\partial_2((d))$}\label{000}
\end{figure}
\newpage
Figure \ref{2-1} and \ref{2-2} shows that $\partial_2\{z^m(e)\}=\partial_2\{-z^m(f)\}$:
\begin{figure}[H]
    \centering
    \includegraphics[width=0.6\linewidth]{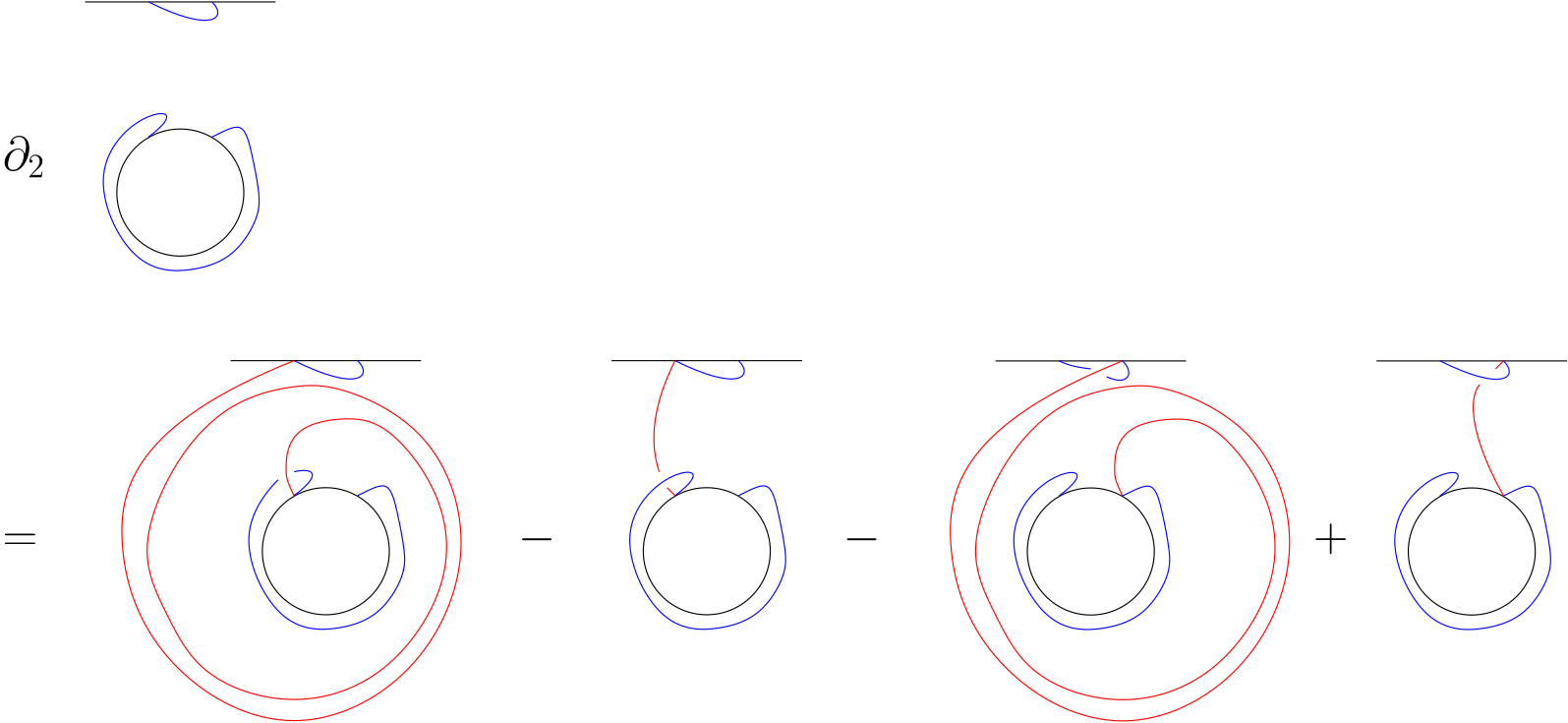}
    \caption{$\partial_2((e))$}\label{2-1}
\end{figure}
\begin{figure}[H]
    \centering
    \includegraphics[width=0.6\linewidth]{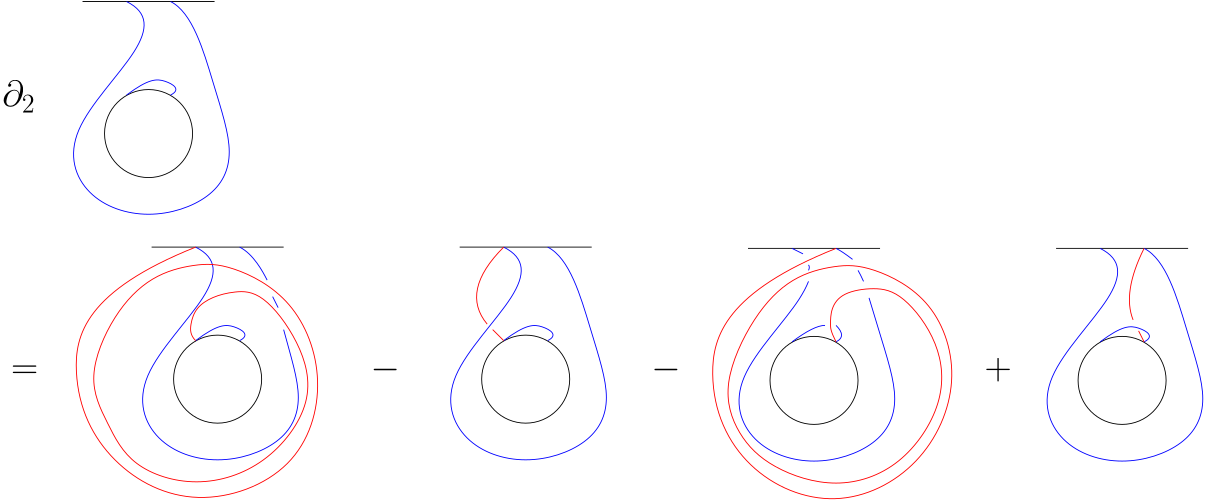}
    \caption{$\partial_2((f))$}\label{2-2}
\end{figure}
Therefore, to compute $\partial_2(C_2(H_1,U))$, we only need to know $\partial_2\{z^m(e)\}$. Compare Figure \ref{z^{2}(e)}, we have: 
\begin{eqnarray*}
    \partial_2(z^k(e))&=& (-A^{-3})c_{-p-1}z^k-(-A^{3})z^kc_{-1}-(-A^3)c_{-p+1}z^k+(-A^{-3})z^kc_{1}\\
    &=& (-A^{-3})(c_{-p-1}z^k+z^kc_{1})-(-A^{3})(c_{-p+1}z^k+z^kc_{-1})
\end{eqnarray*}
From Figure \ref{z^{2}(e)r}, we see that 
$$c_{m}z^k=Ac_{m+1}z^{k-1}+A^{-1}c_{m-1}z^{k-1},$$
and similarly,
$$z^kc_{m}=Az^{k-1}c_{m-1}+A^{-1}z^{k-1}c_{m+1}.$$
These two recursive formula directly lead to Proposition \ref{mk}:
\newpage
\begin{proposition}\label{mk} We have the following formulas,
    \begin{enumerate}
        \item $c_mz^k=\sum\limits_{j=0}^{k}\binom{k}{j}A^{k-2j}c_{m+k-2j}$ and $c_mS_k(z)=\sum_{j=0}^{k} A^{k-2j}\, c_{m + k - 2j}$
        \item $z^kc_m=\sum\limits_{j=0}^{k}\binom{k}{j}A^{k-2j}c_{m-k+2j}$ and $S_k(z)c_m=\sum_{j=0}^{k} A^{k-2j}\, c_{m - k + 2j}$
    \end{enumerate}
\end{proposition}
\begin{proof}

A routing induction will lead to the formulas.
    
\end{proof}

By Proposition \ref{mk}, replace the generators $\{z^k(e)\}$ by $\{S_{k}(e)\}$, we have:
\begin{eqnarray*}
    &&\partial_{2}(S_{k}(z)(e))\\&=&(-A^{-3})(c_{-p-1}S_k(z)+S_k(z)c_{1})-(-A^{3})(c_{-p+1}S_k(z)+S_k(z)c_{-1})\\
    &=&(-A^{-3})\sum_{j=0}^{k} A^{k-2j}\, (c_{-p-1 + k - 2j}+c_{1 - k + 2j})-(-A^{3})\sum_{j=0}^{k} A^{k-2j}\, (c_{-p+1 + k - 2j}+c_{-1 - k + 2j}).
\end{eqnarray*}
\begin{proof}[{\bf proof of Proposition \ref{p1}}]
The computation above tells us that in the first homology group $H_1(H_{1},U)$, the elements represented by cycles in $X_p$ are related. Here we compute some explicit formular for small $k$:
\[
\begin{aligned}
k=0: & -A^{-3}\bigl(c_{-p-1}+c_{1}\bigr) + A^{3}\bigl(c_{-p+1}+c_{-1}\bigr). \\[6pt]
k=1: & \bigl(A^{2}-A^{-2}\bigr)\bigl(c_{-p}+c_{0}\bigr) - A^{-4}\bigl(c_{-p-2}+c_{2}\bigr) + A^{4}\bigl(c_{-p+2}+c_{-2}\bigr). \\[6pt]
k=2: & \bigl(A^{3}-A^{-1}\bigr)\bigl(c_{-p+1}+c_{-1}\bigr) + \bigl(A-A^{-3}\bigr)\bigl(c_{-p-1}+c_{1}\bigr) \\
      & - A^{-5}\bigl(c_{-p-3}+c_{3}\bigr) + A^{5}\bigl(c_{-p+3}+c_{-3}\bigr). \\[6pt]
k=3:& \bigl(A^{4}-1\bigr)\bigl(c_{-p+2}+c_{-2}\bigr) + \bigl(A^{2}-A^{-2}\bigr)\bigl(c_{-p}+c_{0}\bigr) \\
      & + \bigl(1-A^{-4}\bigr)\bigl(c_{-p-2}+c_{2}\bigr) - A^{-6}\bigl(c_{-p-4}+c_{4}\bigr) + A^{6}\bigl(c_{-p+4}+c_{-4}\bigr).
\end{aligned}
\]
For each element $c_{m}+c_{-p-m}$ $X_p$, we give an order to it by $|2m+p|$. We can easily see that, in general, the cycle $c_{-p-1-k}+c_{1+k}$ can be replaced by cycles with smaller order. Thus the only surviving cycles are
\begin{enumerate}
    \item when $p$ is even, $(c_{-p}+c_{0}),(c_{-p+1}+c_{-1}),(c_{-p+2}+c_{-2}),...,(c_{-\frac{p}{2}-1}+c_{-\frac{p}{2}+1}), (c_{-\frac{p}{2}})$
    \item when $p$ is odd, $(c_{-p}+c_{0}),(c_{-p+1}+c_{-1}),(c_{-p+2}+c_{-2}),...,(c_{-\lfloor\frac{p}{2}\rfloor-1}+c_{-\lfloor\frac{p}{2}\rfloor})$
\end{enumerate}
Therefore, the rank of the first homology is $\lfloor\frac{p}{2}\rfloor+1$
\end{proof}

\begin{figure}[H]
    \centering
    \includegraphics[width=0.5\linewidth]{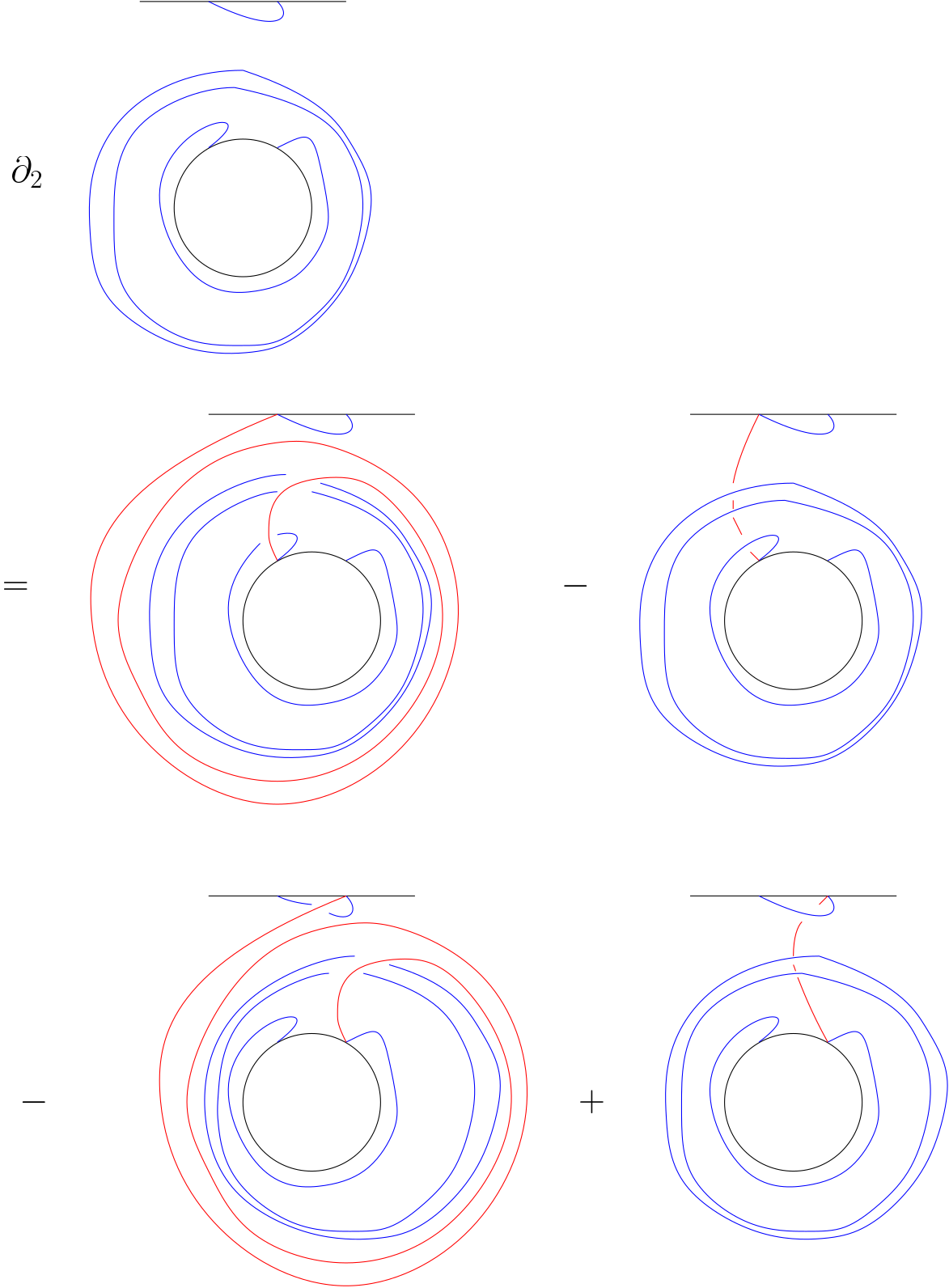}
    \caption{$\partial_2(z^{2}(e))$}\label{z^{2}(e)}
\end{figure}
\begin{figure}[H]
    \centering
    \includegraphics[width=0.5\linewidth]{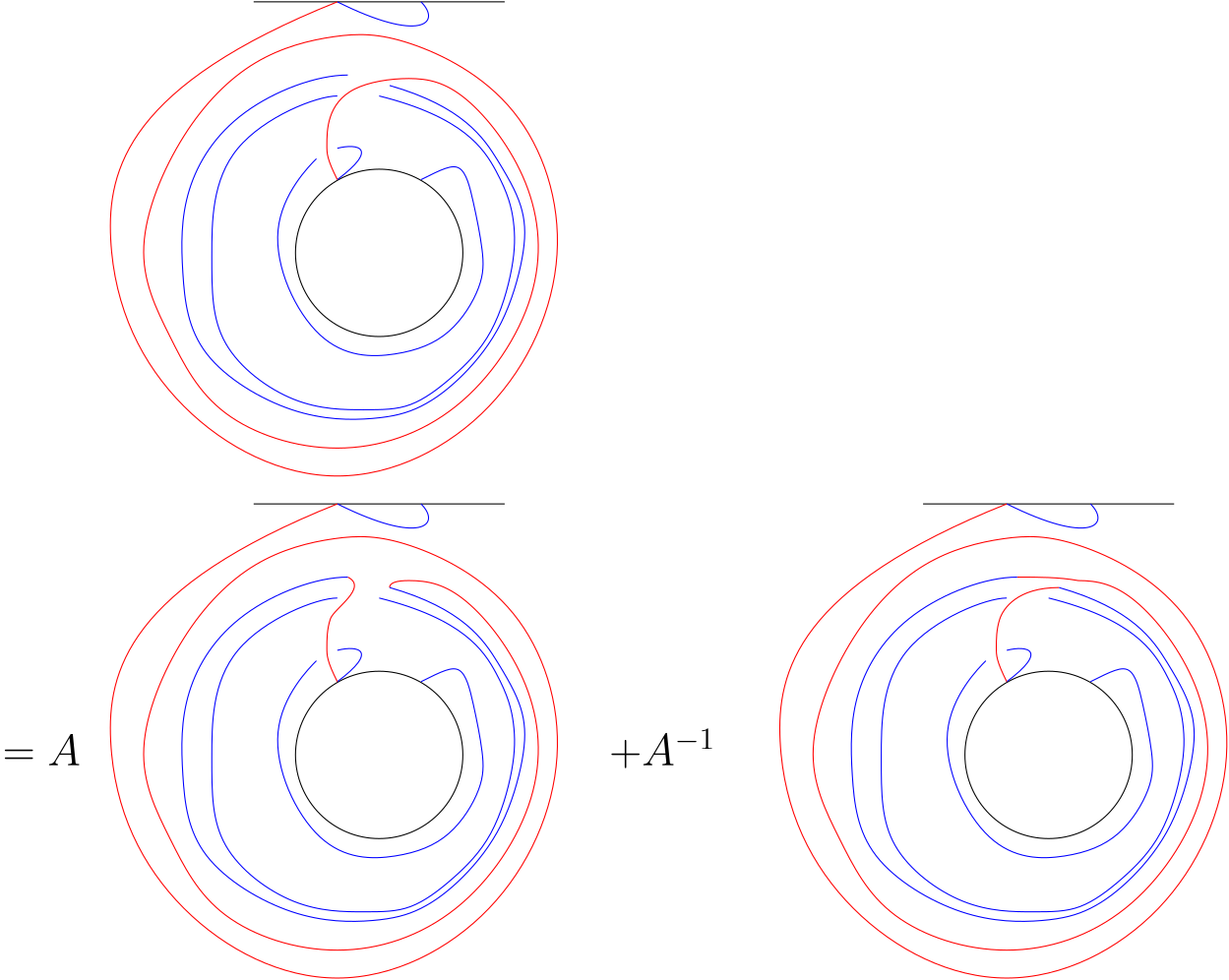}
    \caption{$(-A^{-3})c_{-3}z^2=(-A^{-3})(Ac_{-2}+A^{-1}c_{-4})$}\label{z^{2}(e)r}
\end{figure}

%Consider a filtration $F_p$ on $C_{*}(H_{g_0},U)=\oplus_{n} C_n(H_{g_0},U)$, where $F_pC_{*}(H_{g_0},U)$ is the submodule generated by those relative framed links with $n_1\leq p$. We define a map $f_m:$

‌‌

    \section*{Acknowledgements}

The first author was supported by the National Natural Science Foundation of China (Grant NO. 12471064). The second author was supported by the National Natural Science Foundation of China (Grant No. 11901229, 12371029, 22341304 and W2412041). The authors thank Ruifeng Qiu, Jiajun Wang, Yanqing Zou, and Guannan Guo for helpful and stimulating discussions.


\begin{thebibliography}{9999}

\bibitem[AM]{AM}
M. Abouzaid, C. Manolescu, A sheaf-theoretic model for $SL(2, \mathbb{C})$ Floer
homology, {\it J. Eur. Math. Soc. (JEMS)} 22 (2020), no. 11, 3641–3695.

\bibitem[Bak]{rheasolomarche}
R. P. Bakshi, A counterexample to the generalisation of Witten's conjecture, {\it Contemp. Math.} (to appear). e-print: \href{https://arxiv.org/abs/2205.01653}{arXiv:2205.01653} [math.GT].

\bibitem[BD]{bellettidetcherry}
G. Belletti, R. Detcherry, On torsion in the Kauffman bracket skein module of $3$-manifolds. e-print: \href{https://arxiv.org/abs/2406.17454}{arXiv:2406.17454} [math.GT].

\bibitem[BHM]{BHM} C. Blanchet, N. Habegger, G. Masbaum, P. Vogel, Topological quantum field theories
derived from the Kauffman bracket, {\it Topology} 31(1992), 685–699.
    
\bibitem[BKSW]{BKSW}
Rhea Palak Bakshi, Seongjeong Kim, Shangjun Shi, Xiao Wang,
On the Kauffman bracket skein module of $(S^1\times S^2) \# (S^1\times S^2)$,
Journal of Algebra,
Volume 673,
2025,
Pages 103-137,
ISSN 0021-8693,
https://doi.org/10.1016/j.jalgebra.2025.01.028.


\bibitem[BLP]{blp}
R. P. Bakshi, T. T. Q. Lê, J. H. Przytycki, The Kauffman bracket skein module of the connected sum of two solid tori {\it (draft available on request)}.

\bibitem[BP]{counterhandle}
R. P. Bakshi, J. H. Przytycki, Kauffman bracket skein module of the connected sum of handlebodies: A counterexample, {\it Manuscripta Math.} 167 (2022), no. 3-4, 809–820. \href{https://arxiv.org/abs/2005.07750}{arXiv:2005.07750} [math.GT].

\bibitem[Bul]{sl2cbullock}
D. Bullock,
Rings of $SL_2(\mathbb{C})$-characters and the Kauffman bracket skein module, {\it Comment. Math. Helv.} 72 (1997), no. 4, 521–542. 

\bibitem[BuLo]{knotext}
D. Bullock, W. F. Lo Faro,
The Kauffman bracket skein module of a twist knot exterior. 
{\it Algebr. Geom. Topol.} 5 (2005), 107–118. \href{https://arxiv.org/abs/math/0402102}{arXiv:math/0402102} [math.QA].

\bibitem[BuPr]{smquant}
D. Bullock, J. H. Przytycki, Multiplicative structure of Kauffman bracket skein module quantizations. {\it Proc. Amer. Math. Soc.} 128 (2000), no. 3, 923–931. \href{https://arxiv.org/abs/math/9902117}{arXiv:math/9902117} [math.QA]. 

\bibitem[BW]{BW} F. Bonahon, H. Wong, Representations of the Kauffman bracket skein algebra I: invariants and
miraculous cancellations, {\it Invent. Math.} 204 (1) (2016) 195–243.

\bibitem[DW]{basissurfacetimess1}
R. Detcherry, M. Wolff, A basis for the Kauffman skein module of the product of a surface and a circle. 
{\it Algebr. Geom. Topol.} 21 (2021), no. 6, 2959–2993. \href{https://arxiv.org/abs/2001.05421}{arXiv:2001.05421}[math.GT].

\bibitem[DKS]{dks}
R. Detcherry, E. Kalfagianni, A. S. Sikora, Kauffman bracket skein modules of small $3$-manifolds. e-print: \href{https://arxiv.org/abs/2305.16188}{arXiv:2305.16188} [math.GT]. 

\bibitem[FKL]{FKL}
C. Frohman, J. Kania-Bartoszynska, T. Lê, Unicity for representations of the Kauffman bracket
skein algebra, {\it Invent. Math.} 215 (2) (2019) 609–650.

\bibitem[Gel]{Gel} R. Gelca, Topological quantum field theory with corners based on the Kauffman
bracket, {\it Comment. Math. Helv.}, 72(1997), 216-243.

\bibitem[GJS1]{wittenresolved}
S. Gunningham, D. Jordan, P. Safronov, The finiteness conjecture for skein modules. {\it Invent. Math.} 232 (2023), no. 1, 301–363. \href{https://arxiv.org/abs/1908.05233}{arXiv:1908.05233} [math.QA].

\bibitem[GJS2]{GJS}
Iordan Ganev, David Jordan, Pavel Safronov, The quantum Frobenius for character varieties and multiplicative quiver varieties. {\it J. Eur. Math. Soc.} 27 (2025), no. 7, pp. 3023–3084

\bibitem[GM]{gmsl2cfreegroup}
F. J. González-Acuña, J. M. Montesinos-Amilibia, On the character variety of group representations in $SL(2,C)$ and $PSL(2,C)$. {\it Math. Z.} 214 (1993), no. 4, 627–652.

\bibitem[HP1]{kbsmlens}
J. Hoste, J. H. Przytycki,  The $(2,\infty)$-skein module of lens spaces; a generalization of the Jones polynomial. {\it J. Knot Theory Ramifications} 2 (1993), no. 3, 321–333.

\bibitem[HP2]{s1s2}
J. Hoste, J. H. Przytycki, The Kauffman bracket skein module of $S^1 \times S^2$. {\it Math. Z.} 220 (1995), no. 1, 65–73.

\bibitem[Kau]{sm&j}
L. H. Kauffman, State models and the Jones polynomial. {\it Topology}, 26 (1987), no. 3, 395-407.

\bibitem[KL]{KL} L. Kauffman, S. Lins, Temperly Lieb Recoupling Theory and Invariants of 3
Manifolds, {\it Annals of Mathematics Studies}, No. 134, Princeton University Press,
Princeton, New Jersey, 1994

%\bibitem[McL]{mclendon}
%M. McLendon,
%Detecting torsion in skein modules using Hochschild homology. 
%{\it J. Knot Theory Ramifications} 15 (2006), no. 2, 259–277. e-print: \href{https://arxiv.org/abs/math/0405395}{arXiv:math/0405395} [math.GT]. 

\bibitem[Lic]{Lic} W. B. R. Lickorish, The skein method for three-manifold invariants, {\it J. Knot Theor.
Ramif.}, 2(1993) no. 2, 171–194.

\bibitem[Prz1]{smof3}
J. H. Przytycki, Skein modules of 3-manifolds. {\it Bull. Polish Acad. Sci. Math.} 39 (1991), no. 1-2, 91–100. \href{https://arxiv.org/abs/math/0611797}{arXiv:math/0611797}[math.GT].

%\bibitem[Prz2]{algtop}
%J. H. Przytycki, Algebraic topology based on knots: an introduction. {\it KNOTS '96 (Tokyo).} 279–297, World Sci. Publ., River Edge, NJ, 1997.

\bibitem[Prz2]{fundamentals}  
J. H. Przytycki, Fundamentals of Kauffman bracket skein modules. {\it Kobe Math. J.}, 16(1), 1999, 45-66. \href{https://arxiv.org/abs/math/9809113}{arXiv:math/9809113} [math.GT]. 

\bibitem[Prz3]{connsum}
J. H. Przytycki,
Kauffman bracket skein module of a connected sum of $3$-manifolds.  
{\it Manuscripta Math.} 101 (2000), no. 2, 199–207. \href{https://arxiv.org/abs/math/9911120}{arXiv:math/9911120} [math.GT].

\bibitem[PS]{saofsurfaces}
J. H. Przytycki, A. S. Sikora,
Skein algebras of surfaces. {\it Trans. Amer. Math. Soc.} 371 (2019), no. 2, 1309–1332. \href{https://arxiv.org/abs/1602.07402}{arXiv:1602.07402}[math.QA].

\bibitem[Rob]{Rob} J. Roberts, Skeins and mapping class groups, {\it Math. Proc. Cambridge Phil. Soc.},
115(1994), 53–77.

\bibitem[Tur1]{turaevsolidtorus}
V. G. Turaev,
The Conway and Kauffman modules of a solid torus.
{\it Zap. Nauchn. Sem. Leningrad. Otdel. Mat. Inst. Steklov.} (LOMI) 167 (1988), Issled. Topol. 6, 79–89, 190; translation in {\it J. Soviet Math.} 52 (1990), no. 1, 2799–2805.

\bibitem[Tur2]{Tur2}V. G. Turaev, Algebras of loops on surfaces, algebras of knots, and quantization, {\it Adv.
Ser. in Math.} Physics 9 (1989), ed. C. N. Yang, M. L. Ge, 59–95.

%\bibitem[Wit]{wittenyangmills}
%E. Witten,
%On quantum gauge theories in two dimensions.
%{\it Comm. Math. Phys.} 141 (1991), no. 1, 153–209.

\end{thebibliography}
\end{document}